\documentclass[11pt,a4paper,reqno]{amsart}
\usepackage[english]{babel}
\usepackage[applemac]{inputenc}
\usepackage[T1]{fontenc}
\usepackage{palatino}
\usepackage{verbatim}
\usepackage{amsmath}
\usepackage{amssymb}
\usepackage{amsthm}
\usepackage{amsfonts}
\usepackage{graphicx}
\usepackage{mathtools}
\usepackage{enumitem}

\usepackage[colorlinks = true, citecolor = black]{hyperref}
\author{David Bate and Tuomas Orponen}
\title{On the conformal dimension of product measures}
\address{University of Helsinki, Department of Mathematics and Statistics}
\thanks{Both D.B. and T.O. are supported by the Academy of Finland via the Centre of Excellence in Analysis and Dynamics Research (project No. 307333). D.B. is supported by the University of Helsinki via the project Quantitative rectifiability of sets and measures in Euclidean Spaces and Heisenberg groups (project No. 7516125). T.O. is supported by the Academy of Finland via the project Restricted families of projections and connections to Kakeya type problems, grant No. 274512.}
\email{david.bate@helsinki.fi \\ tuomas.orponen@helsinki.fi}

\newcommand{\R}{\mathbb{R}}

\newcommand{\N}{\mathbb{N}}

\newcommand{\Z}{\mathbb{Z}}
\newcommand{\tn}{\mathbb{P}}
\newcommand{\calJ}{\mathcal{J}}

\newcommand{\calL}{\mathcal{L}}
\newcommand{\calD}{\mathcal{D}}
\newcommand{\calH}{\mathcal{H}}

\newcommand{\calB}{\mathcal{B}}
\newcommand{\calG}{\mathcal{G}}

\newcommand{\calI}{\mathcal{I}}
\newcommand{\calS}{\mathcal{S}}

\newcommand{\spt}{\operatorname{spt}}
\newcommand{\Hd}{\dim_{\mathrm{H}}}

\newcommand{\E}{\mathbb{E}}

\newcommand{\diam}{\operatorname{diam}}
\newcommand{\card}{\operatorname{card}}
\newcommand{\dist}{\operatorname{dist}}

\newcommand{\gen}{\mathrm{gen}}
\newcommand{\Cd}{\dim_{\mathrm{C}}}
\newcommand{\g}{g}

\numberwithin{equation}{section}

\theoremstyle{plain}
\newtheorem{thm}[equation]{Theorem}

\newtheorem{lemma}[equation]{Lemma}

\newtheorem{ex}[equation]{Example}

\newtheorem{proposition}[equation]{Proposition}
\newtheorem{question}{Question}

\theoremstyle{definition}

\newtheorem{definition}[equation]{Definition}

\newtheorem{notation}[equation]{Notation}

\theoremstyle{remark}
\newtheorem{remark}[equation]{Remark}

\newcommand{\nref}[1]{(\hyperref[#1]{#1})}

\begin{document}

\begin{abstract} Given a compact set $E \subset \R^{d - 1}$, $d \geq 1$, write $K_{E} := [0,1] \times E \subset \R^{d}$. A theorem of C. Bishop and J. Tyson states that any set of the form $K_{E}$ is \emph{minimal for conformal dimension}: if $(X,d)$ is a metric space and $f \colon K_{E} \to (X,d)$ is a quasisymmetric homeomorphism, then 
\begin{displaymath} \Hd f(K_{E}) \geq \Hd K_{E}. \end{displaymath}
We prove that the measure-theoretic analogue of the result is not true. For any $d \geq 2$ and $0 \leq s < d - 1$, there exist compact sets $E \subset \R^{d - 1}$ with $0 < \calH^{s}(E) < \infty$ such that the conformal dimension of $\nu$, the restriction of the $(1 + s)$-dimensional Hausdorff measure on $K_{E}$, is zero. More precisely, for any $\epsilon > 0$, there exists a quasisymmetric embedding $F \colon K_{E} \to \R^{d}$ such that $\Hd F_{\sharp}\nu < \epsilon$.
\end{abstract}

\maketitle


\section{Introduction}

We start by recalling the notions of \emph{quasisymmetric maps} and \emph{conformal dimension}. 
\begin{definition}[Quasisymmetric maps] A map $f \colon (X,d) \to (Y,d')$ between two metric spaces $(X,d),(Y,d')$ is called \emph{quasisymmetric}, if there is a homeomorphism $\eta \colon [0,\infty) \to [0,\infty)$ such that the inequality
\begin{displaymath} \frac{d'(f(x),f(a))}{d'(f(x),f(b))} \leq \eta \left(\frac{d(x,a)}{d(x,b)} \right) \end{displaymath}
holds for all triples $x,a,b \in X$ with $x \neq b$.
\end{definition}
\begin{definition}[Conformal dimension]\label{Cdim} The \emph{conformal dimension} of a metric space $(X,d)$ is 
\begin{displaymath} \Cd X = \inf_{f} \Hd f(X), \end{displaymath}
where $\Hd$ stands for Hausdorff dimension, and the $\inf$ is taken over all quasisymmetric homeomorphisms $f$ between $(X,d)$ and any metric space $(Y,d')$. The space $(X,d)$ is called \emph{minimal for conformal dimension}, if $\Cd X = \Hd X$. 
\end{definition}
The notion of conformal dimension was first introduced by Pansu \cite{Pa} in 1989. For an extensive introduction to the subject, and plenty of additional references, see the monograph \cite{MT} of Mackay and Tyson. A lower bound for $\Cd X$ is the topological dimension of $X$, namely the $\inf$ of the dimensions of metric spaces homeomorphic to $X$. Thus, for example, $\Cd [0,1] = 1$. A well-known heuristic suggests that an improvement over the trivial bound can be expected, if $X$ contains a sufficiently rich family of connected subsets. As far as we know, the principle first appeared in Pansu's work, \cite[Proposition 2.9]{Pa}, and is, today, supported by a large body of specific results, see \cite[4.6 Notes]{MT}. For the motivation of this paper, the following result is most relevant. It appeared implicitly in the 2001 paper \cite{BT} of Bishop and Tyson, and is stated explicitly as \cite[Proposition 4.1.11]{MT}:
\begin{thm}\label{BTTheorem} If $E \subset \R^{d - 1}$ is compact, $d \geq 1$, then $[0,1] \times E \subset \R^{d}$ is minimal for conformal dimension. 
\end{thm}

For the most general and recent results in this vein, see Section 4 in the paper \cite{BHW} of Bishop, Hakobyan and Williams. 

A natural counterpart for the conformal dimension of a metric space is the \emph{conformal dimension of a measure}:
\begin{definition}[Conformal dimension of measures]\label{CdimMu} Let $(X,\mu,d)$ be a metric measure space, where $\mu$ is a locally finite Borel measure. The \emph{conformal dimension of $\mu$} is the number
\begin{displaymath} \Cd \mu := \inf_{f} \Hd f_{\sharp}\mu, \end{displaymath}
where the $\inf$ ranges over all quasisymmetric homeomorphisms between $(\spt \mu,d)$ and any metric space $(Y,d')$. Here $\Hd \mu := \inf \{\Hd A : \mu(A^{c}) = 0\}$, and $f_{\sharp}\mu$ is the push-forward of $\mu$ under $f$, defined by $(f_{\sharp}\mu)(A) = \mu(f^{-1}(A))$ for Borel sets $A$. A measure $\mu$ is \emph{minimal for conformal dimension}, if $\Cd \mu = \Hd \mu$.
\end{definition}

As discussed above, the conformal dimension of metric space $(X,d)$ is intricately related to the topology and connectivity of $X$. In this light, even considering the conformal dimension of a measure $\mu$ on $X$ may seem unnatural: in order to prove that $\Cd \mu$ is small, it suffices to find a set $A$ of full $\mu$-measure, and a quasisymmetric homeomorphism $f \colon \spt \mu \to (Y,d')$ such that $\Hd f(A)$ is small. From a topological viewpoint, the set $A$ may easily be far smaller than $\spt \mu$. So, it is not at all clear to begin with, if the structure of $\spt \mu$ -- topological or metric -- plays any role in the problem. In fact, an example of Tukia \cite{Tu} from 1989 seems to suggest that it does not:
\begin{ex}[Tukia]\label{tukiaEx} For any $\epsilon > 0$, there is a Borel subset $B \subset [0,1]$ of Lebesgue measure unity, and a quasisymmetric homeomorphism $f \colon [0,1] \to [0,1]$ such that $\Hd f(B) < \epsilon$. In particular, the conformal dimension of Lebesgue measure on $[0,1]$ equals zero. 
\end{ex}

We remark that the sets $B \subset [0,1]$ in Example \ref{tukiaEx} are far from arbitrary. In fact, under suitable uniform "fatness" assumptions on a (compact, totally disconnected) set $B \subset [0,1]$ of positive measure, no quasisymmetric homeomorphism can map $B$ to a null set, let alone lower its dimension. For more information, see Staples and Ward \cite{SW}, and Buckley, Hanson, and MacManus \cite{BHM}.

Analogous problems become much harder for measures in $[0,1]^{d}$, for $d \geq 2$. Question 16 on the list of Heinonen and Semmes \cite{HS} asks the following: if $X$ is a metric space and $f \colon X \to \R^{d}$ is a quasisymmetric homeomorphism, $d \geq 2$, then does $f$ map sets of null $\calH^{d}$ measure to sets of null $\calH^{d}$ measure? Question 15 on the list asks the same, assuming \emph{a priori} $X$ to have locally finite $\calH^{d}$ measure. The questions are open even if $d = 2$ and $X \subset \R^{3}$.

These problems can be (very nearly) re-phrased in terms of the conformal dimension of the Lebesgue measure $\mathcal{L}^{d}$ on $\R^{d}$, $d \geq 2$. Namely, if $\Cd \mathcal{L}^{d} < d$, then, by definition, there exists a metric space $X$ and a quasisymmetric homeomorphism $f \colon \R^{d} \to X$ such that $\Hd f(A) < n$ for some set $A \subset \R^{d}$ of positive (or even full) measure. This would imply that the quasisymmetric homeomorphism $f^{-1} \colon X \to \R^{d}$ sends the $\calH^{d}$ null set $f(A) \subset X$ to a set of positive $d$-dimensional measure, answering Question 16 in the negative.

The purpose of the current paper is to investigate the situation in-between Tukia's example, and the Heinonen-Semmes problems. What if $f$ is a quasisymmetric homeomorphism defined on a set of the form $[0,1] \times E$, where $E$ has many more points than one (Tukia's example), but not quite Lebesgue positively many of them (Heinonen-Semmes problems)? Recalling the result of Bishop and Tyson, Theorem \ref{BTTheorem}, this seems like a very natural intermediate question.

Here is the main result of the paper:
\begin{thm}\label{mainHigherDimension} For any $d \geq 2$ and $0 \leq t < d - 1$, there exist compact sets $E \subset \R^{d - 1}$ with $0 < \calH^{t}(E) < \infty$ such that the conformal dimension of the $(1 + t)$-dimensional Hausdorff measure $\nu$ on $[0,1] \times E$ is zero: for any $\epsilon > 0$, there exists a quasisymmetric embedding $F \colon [0,1] \times E \to \R^{d}$ such that $\Hd F_{\sharp}\nu < \epsilon$.
\end{thm}

In fact, our proof gives something slightly stronger. For brevity, we denote the restriction of one-dimensional Lebesgue measure to $[0,1]$ by $\calL$. 
\begin{thm}\label{mainTechnical} For $d \geq 2$, there exists a dense set of values $s \in (0,d - 1)$, and compact sets $E \subset \R^{d - 1}$ with $0 < \calH^{s}(E) < \infty$, and with the following property. If $\nu$ is any Radon measure supported on $E$, then the conformal dimension of $\calL \times \nu$ is zero. In fact, for any $\epsilon > 0$, there exists a quasisymmetric embedding $F \colon [0,1] \times E \to \R^{d}$ such that $\Hd F_{\sharp}(\calL \times \nu) < \epsilon$ simultaneously for all Radon measures $\nu$ supported on $E$. 
\end{thm}
Theorem \ref{mainTechnical} easily implies Theorem \ref{mainHigherDimension}: if $0 \leq t < d - 1$ is as in Theorem \ref{mainHigherDimension}, one can pick $t < s < d - 1$ and $E$ as in Theorem \ref{mainTechnical}, with $0 < \calH^{s}(E) < \infty$. Then, one can find a subset $E' \subset E$ with $0 < \calH^{t}(E') < \infty$ (see for instance Theorem 8.13 in \cite{Ma}), and apply Theorem \ref{mainTechnical} to $\calL \times \nu$ with $\nu := \calH^{t}|_{E'}$. Of course, one still needs to check that $\calL \times \nu$ is equivalent to the $(1 + t)$-dimensional Hausdorff measure on $[0,1] \times E'$; this follows from the work of Howroyd \cite{Ho}, for instance, but we also include the details in Appendix \ref{A}.

To emphasise how extremely poorly the notion of conformal dimension of measures is understood, we conclude the introduction with two questions:
\begin{question} Do there exist measures with positive and finite conformal dimension? \end{question}
For measures supported on $\R$, the answer is negative, see \cite{O}. The set $E$ constructed in Theorem \ref{mainHigherDimension} is certainly not $s$-Ahlfors-David regular, and there is a clear obstruction, why our construction could not work in that situation. So, the following particular case of the previous question seems particularly compelling:

\begin{question}\label{KSQuestion} Let $C \subset \R$ be the middle-thirds Cantor set of dimension $s = \log 2/\log 3$, and let $\mu$ be the $(1 + s)$-dimensional Hausdorff measure on $[0,1] \times C$. Is $\mu$ minimal for conformal dimension?
\end{question}

Question \ref{KSQuestion} was proposed to us by A. K\"aenm\"aki and T. Sahlsten, and it served as an initial motivation for this paper. 

Finally, we remark that the dimension $\Hd \mu$ is sometimes referred to as the \emph{upper Hausdorff dimension} of $\mu$, whereas the \emph{lower Hausdorff dimension} is $\inf \{\Hd A : \mu(A) > 0\}$, which is bounded above by the upper Hausdorff dimension. The main result of the paper remains valid, and the questions stated above remain reasonable, if the reader prefers the latter definition for $\Hd \mu$. For more information on various dimensions of measures, see Section 10 in Falconer's book \cite{Fa2}.

\subsection{Notation} For $x \in \R^{d}$ and $r > 0$, the notation $B(x,r)$ stands for a closed Euclidean ball centred at $x$, with radius $r > 0$. The symbol $\calH^{t}$ stands for $t$-dimensional Hausdorff measure in $\R^{d}$, where the dimension "$d$" of the ambient space should always be clear from the context. Lebesgue measure on $\R^{d}$ is denoted by $\mathcal{L}^{d}$; the restriction of $\calL^{1}$ to $[0,1]$ is further abbreviated to $\calL$. If $A,B \geq 0$, the notation $A \lesssim_{p} B$ means that $A \leq CB$ for a constant $C \geq 1$, which only depends on the parameter $p$; if no such parameter $p$ is specified, then the constant $C$ is absolute (unless otherwise stated). The notation $A \sim_{p} B$ is shorthand for $A \lesssim_{p} B \lesssim_{p} A$, and $A \gtrsim_{p} B$ means the same as $B \lesssim_{p} A$.

\subsection{Outline of the proof in the case $d = 2$}\label{outline} The easiest part of the construction is finding a suitable compact set $E \subset [-\tfrac{1}{2},\tfrac{1}{2}] \subset \R^{d - 1}$ with $0 < \calH^{s}(E) < \infty$. As we pointed out above, an $s$-Ahlfors-David regular choice of $E$ would not work for the other other parts of the construction, but there are virtually no other restrictions: a very generic Cantor-type construction works, as long as the "branching" is sufficiently rapid.

Assume that $E$ has been constructed, as above, and write $K := [0,1] \times E \subset \R^{2}$. To prove Theorem \ref{mainTechnical}, it suffices to pick $\epsilon > 0$ and construct a quasisymmetric embedding $F \colon K \to \R^{2}$ such that $\Hd F_{\sharp}(\calL \times \nu) < \epsilon$ for all Radon measures $\nu$ supported on $E$. The mapping $F$ will have the form
\begin{displaymath} F(x,y) = (f(x),\g(x,y)), \end{displaymath}
where $f$ is quasisymmetric homeomorphism $[0,1] \to [0,1]$ with $f_{\sharp}\calL < \epsilon$. An instance of such a map is given Tukia's example, but we have to be significantly more careful with the construction. The main challenge of the proof is finding a "conjugate" map $\g \colon K \to \R$, which makes $F$ quasisymmetric on $K$. Products of quasisymmetric maps are usually far from quasisymmetric, and taking $\g(x,y) = y$ fails spectacularly. In fact, this choice would also be inadequate in the sense that $\Hd F_{\sharp}(\calL \times \nu) \geq \Hd \nu$, whereas Theorem \ref{mainTechnical} requires $\Hd F_{\sharp}(\calL \times \nu)$ to be arbitrarily close to zero, independently of $\nu$.

It turns out that if $E$ is sufficiently far from $s$-Ahlfors-David regular, then $\g$ can be defined so that $F$, as above, is a quasisymmetric embedding of $K$, and moreover $\g$ satisfies the inequality
\begin{equation}\label{form58} |\g(x,y) - \g(x',y)| \lesssim |f(x) - f(x')|, \qquad (x,y),(x',y) \in K. \end{equation}
This implies, rather easily, that $F$ distorts the dimension of $\calL \times \nu$ by about as much as $f$ distorts the dimension of $\calL$. This will complete the proof of Theorem \ref{mainTechnical}. 

If $E$ were $s$-Ahlfors-David regular, achieving \eqref{form58} seems very difficult, unless $f$ is absolutely continuous with $f' \in A_{\infty}$, see Remark \ref{aInftyRemark}. But then $f$ would not lower the dimension of $\calL$, and $F$ would not lower the dimension of $\calL \times \nu$. So, if the answer to Question \ref{KSQuestion} is negative, then the counter-example most likely has to look quite different from the one in this paper. 

\subsection{Acknowledgements} We initially heard Question \ref{KSQuestion} from Tuomas Sahlsten in Spring 2016, but it is due to both K\"aenm\"aki and Sahlsten. Until late January 2017, we were unaware of the result of Bishop and Tyson, Theorem \ref{BTTheorem}: we believed that, for positive results, one needs to assume that $E$ is $s$-Ahlfors-David regular, and without that hypothesis, it may happen that $\Cd ([0,1] \times E) = 0$ (and in particular $\Cd \calH^{1 + s}|_{[0,1] \times E} = 0$). So, we spent several weeks trying to prove the Ahlfors-David regular variant of Theorem \ref{mainHigherDimension}, which we thought was the only non-trivial question around. In late January 2017, we heard a talk of K\"aenm\"aki at the University of Helsinki, where he used Bishop and Tyson's theorem in full generality. After a short phase of disbelief, and discussions with K\"aenm\"aki, we realised that our attempts, which did not work for Ahlfors-David regular sets $E$, sufficed to settle the non-regular variant of the problem. So, we are most grateful to K\"aenm\"aki: first for inventing the hard problem (Question \ref{KSQuestion}) with Sahlsten, second for pointing out the easier variant (Theorem \ref{mainHigherDimension}), which we could actually solve, and third for useful discussions.

We are also grateful to M. Romney for making us aware of the relationship between this problem and the questions of Heinonen and Semmes, and to V. Chousionis and K. F\"assler for fruitful discussions.

\section{Constructions} \label{constructions}

\subsection{Construction of the set}\label{setConstruction} In this section, we review an entirely standard construction of a (non Ahlfors-David-regular) Cantor-type set $E \subset [-\tfrac{1}{2},\tfrac{1}{2}]$. The letter "$E$" will always be reserved for this subset of $\R$, and the set "$E \subset \R^{d - 1}$" appearing in Theorems \ref{mainHigherDimension} and \ref{mainTechnical} will, in fact, be the $(d - 1)$-fold product $E \times \ldots \times E = E^{d - 1}$. 

\begin{definition}[The set $E$] For any dyadic number $s = p_{s}/2^{m_{s}} \in (0,1)$, with $0 < p_{s} < 2^{m_{s}}$, a set $E = E_{s} \subset [-\tfrac{1}{2},\tfrac{1}{2}]$ will next be constructed via an iterative procedure, so that eventually $0 < \calH^{s}(E) < \infty$. The representation $s = p_{s}/2^{m_{s}}$ is not unique, and it will sometimes be convenient to assume that $p_{s},m_{s}$ are large integers.

Let $\calI_{0} = \{[-\tfrac{1}{2},\tfrac{1}{2}]\}$, and assume that a collection $\calI_{n}$ of closed sub-intervals of $[-\tfrac{1}{2},\tfrac{1}{2}]$ has already been defined. Write $r_{0} := 1$, and assume that, for $n \geq 1$, all the intervals in $\calI_{n}$ have equal length
\begin{displaymath} |I_{n}| = r_{n} := 3^{-2^{n + m_{s}}} \in 3^{-\N}. \end{displaymath}
Everything that follows would work equally well for any sequence $(r_{n})_{n \in \N}$ of numbers in $3^{-\N}$ with sufficiently rapid decay. Next, define the integer sequence $(m_{n})_{n \in \N}$ by requiring that
\begin{equation}\label{mrn} m_{n + 1}(r_{n + 1})^{s} = r_{n}^{s}. \end{equation}
This is possible, because, recalling that $s = p_{s}/2^{m_{s}}$,
\begin{displaymath} \frac{r_{n}^{s}}{(r_{n + 1})^{s}} = \frac{3^{-s \cdot 2^{m_{s} + n}}}{3^{-s \cdot 2^{m_{s} + n + 1}}} = 3^{(p_{s}/2^{m_{s}}) \cdot 2^{m_{s} + n}} = 3^{p_{s} \cdot 2^{n}} \in \N. \end{displaymath}
Then, define $\calI_{n + 1}$ by placing $m_{n + 1}$ equally spaced closed intervals of length $r_{n + 1}$ inside every interval of $I \in \calI_{n}$, so that this spacing is as large as possible. The spacing of consecutive intervals in $\calI_{n + 1}$ will be denoted by 
\begin{displaymath} s_{n + 1} := \min\{\dist(I_{1},I_{2}) : I_{1},I_{2} \in \calI_{n + 1} \text{ are distinct}\}, \quad n \geq 0. \end{displaymath}
Then, by \eqref{mrn},
\begin{equation}\label{rnsn} s_{n + 1} \sim \frac{r_{n}}{m_{n + 1}} = r_{n} \cdot \frac{(r_{n + 1})^{s}}{r_{n}^{s}} = r_{n} \cdot 3^{-p_{s} \cdot 2^{n}},  \end{equation}
so $s_{n + 1} < r_{n}/10$, if $p_{s} \geq 3$. On the other hand, since the sequence $(r_{n})_{n \in \N}$ decays rapidly, and $s < 1$, the spacing $s_{n + 1}$ is significantly larger than the length $r_{n + 1}$:
\begin{displaymath} \frac{r_{n + 1}}{s_{n + 1}} \sim m_{n + 1}\frac{r_{n + 1}}{r_{n}} = \left(\frac{r_{n + 1}}{r_{n}} \right)^{1 - s} = 3^{-(1 - s) \cdot 2^{m_{s} + n}}. \end{displaymath}
With this in mind, we choose $m_{s}$ so large that the ratios $r_{n + 1}/s_{n + 1}$ are uniformly bounded by $3^{-(n+1)}$. In summary,
\begin{equation}\label{form50} 3^{n+1}r_{n + 1} < s_{n + 1} < r_{n}/10, \quad n \geq 0. \end{equation}

Once all the collections $\calI_{n}$ have been constructed in this way, write
\begin{displaymath} E_{n} := \bigcup_{I \in \calI_{n}} I, \end{displaymath}
and $E := \bigcap E_{n}$. This completes the construction of $E$. Based on \eqref{mrn}, it is standard to check (or see Theorem 1.15 in Falconer's book \cite{Fa}) that the $E$ has positive and finite $s$-dimensional Hausdorff measure, and the $(d - 1)$-fold product $E^{d - 1}$ has positive and finite $s(d - 1)$-dimensional Hausdorff measure. So, by varying $s = p_{s}/2^{m_{s}} \in (0,1)$, the Hausdorff dimension of $E^{d - 1}$ attains values arbitrarily close to $d - 1$, as required by Theorem \ref{mainTechnical}.
\end{definition}

We conclude the section by introducing some additional notation:

\begin{definition}[Parents in $Y_{n}$]\label{parents} Recall the collection of intervals $\calI_{n}$ from the construction above, and let $Y_{n}$ be the set of all midpoints of intervals in $\calI_{n}$. For $y \in E$ and $n\in \N$, there exists a unique interval $I \in \calI_{n}$ containing $y$.  The \emph{level $n$ parent of $y$} is the midpoint of $I$ and is denoted by $y_{n} := y_{n}(y) \in Y_{n}$. 
Since $y_{n}$ and $y$ both belong to $I$, one has
\begin{displaymath} |y - y_{n}| \leq |I| = r_{n}. \end{displaymath}
\end{definition}

\subsection{Construction of the measure}\label{construction-measure} In this section, we construct a special doubling measure $\mu$ on the real line. This measure is associated to the quasisymmetric homeomorphism "$f$" from Section \ref{outline}. 

For the rest of the paper, we now fix a small number $\epsilon > 0$, and write $K := K^{d}_{E} := [0,1] \times E^{d - 1}$. To prove Theorem \ref{mainTechnical}, we need to construct a quasisymmetric embedding $F = F_{\epsilon} \colon K \to \R^{d}$ such that
\begin{displaymath} \Hd F_{\sharp}(\calL \times \nu) < \epsilon \end{displaymath}
for all Radon measures $\nu$ supported on $E$. Note that $F$ depends on $\epsilon$, while $K$ does not. We start by constructing a suitable doubling measure $\mu = \mu_{\epsilon}$ on the real line. Let $\calD$ be the collection of all \emph{ternary} intervals of $\R$, with length at most one:
\begin{displaymath} \calD = \bigcup_{k \in \N} \calD_{k}, \end{displaymath}
where $\calD_{0} = \{[j,j + 1) : j \in \Z\}$, and the intervals in $\calD_{k + 1}$ are obtained by partitioning each interval in $\calD_{k}$ into three half-open intervals of equal length. Then, for $n \in \N$, including $n = 0$, let $\calS_{n}$ be the ternary intervals of length $r_{n}$, where $r_{n} \in 3^{-\N}$ was defined in the previous section. So, formally,
\begin{equation}\label{calSn} \calS_{n} := \calD_{-\log_{3} r_{n}}. \end{equation}
We also write $\calS := \bigcup_{n \geq 0} \calS_{n}$. Now, following the paper \cite{GKS} of Garnett, Killip and Schul, we define the function $h \colon \R \to \{-1,2\}$ by
\begin{displaymath} h(x) := \begin{cases} 2, & \text{if } x \in [\tfrac{1}{3},\tfrac{2}{3}) + \Z, \\ -1, & \text{otherwise.} \end{cases} \end{displaymath}
Note that $h$ is constant on elements of $\calD_{1}$, and this is the only reason why ternary intervals are considered in this paper, instead of dyadic ones. The measure $\mu$ will be defined as a weak limit of the partial "Riesz products"
\begin{equation} \mu_{n} = \mathop{\prod_{j \in \calJ}}_{j < n} (1 + \alpha \cdot h(3^{j}x)) \, dx, \qquad n \geq 1, \end{equation}
where $\alpha = \alpha_{\epsilon,s} \in [0,1)$ is a suitable constant, and $\calJ = \calJ_{\epsilon} \subset \{0,\ldots,n\}$ is a non-empty collection of indices (both $\alpha$ and $\calJ$ will be specified later; $\alpha$ will be chosen quite late in \eqref{form45}, whereas $\calJ$ is specified in this section). Regardless of $\calJ$, the choice $\alpha = 0$ gives Lebesgue measure; also note that, for a fixed $n \in \N$, the measure $\mu_{n}$ is just a function, which is constant on the ternary intervals in $\calD_{n}$. The value of this constant is given by the common value of
\begin{displaymath} x \mapsto \mathop{\prod_{j \in \calJ}}_{j < n} (1 + \alpha \cdot h(3^{j}x)) \end{displaymath}
 on the interval $I$, which will be denoted by $\pi(I) = \pi_{\alpha,\calJ}(I)$. Then
 \begin{equation}\label{densityFormula} \Theta(I) := \frac{\mu(I)}{|I|} = \pi(I). \end{equation}
 The proof in \cite{GKS} shows that
\begin{equation}\label{form33} \mu = \lim_{n \to \infty} \mu_{n} := \prod_{j \in \calJ} (1 + \alpha \cdot h(3^{j}x)) \, dx \end{equation}
exists, and is a doubling measure, with doubling constant depending only on how close $\alpha$ is to $1$ (and not on the choice of $\calJ$). The details are given below for the convenience of the reader.
\begin{lemma}\label{uniform-doubling} The measure $\mu$ is doubling with a constant $D = D_{\alpha} \geq 1$, which is independent on the choice of $\calJ$.
\end{lemma}

\begin{proof} It clearly suffices to prove that
\begin{equation}\label{form49} \frac{\mu(I_{1})}{\mu(I_{2})} \leq \frac{1 + 2\alpha}{1 - \alpha} \end{equation}
for every pair of adjacent ternary intervals $I_{1},I_{2}$. For ternary intervals of unit length, this is clear, because $\mu([j,j + 1)) = 1$ for $j \in \Z$. Next, fix $n \geq 1$. Every ternary interval $I=[x,x+3^{-n}) \in \calD_{n}$ has a unique representation $(x_{0}(I),x_{1}(I),\ldots,x_{n}(I))$, where $x_{0}(I) = [x]$ is the integer part of $x$ and $0.x_{1}(I)x_{2}(I)\ldots x_{n}(I)$ is the base 3 representation of $x-[x]$. For $I \in \calD_{n}$, write
\begin{displaymath} 1(I) := \card\{1 \leq j \leq n : x_{j}(I) = 1\}. \end{displaymath}
Now, if $I_{1},I_{2} \in \calD_{n}$ are adjacent, we claim that $|1(I_{1}) - 1(I_{2})| \leq 1$. Indeed, if $j_{0} \geq 1$ is the first index such that $x_{j_{0}}(I_{1}) = 1$ and $x_{j_{0}}(I_{2}) = 0$, say, then the adjacency of $I_{1}$ and $I_{2}$ forces
\begin{displaymath} x_{j}(I_{1}) = 0 \quad \text{and} \quad x_{j}(I_{2}) = 2 \end{displaymath}
for $j_{0} < j \leq n$. This proves the claim. Next, let
\begin{displaymath} \calJ_{n} := \calJ \cap \{1,\ldots,n\}, \end{displaymath}
and for $I \in \calD_{n}$, write $1_{\calJ}(I) := \card\{1 \leq j \leq n : j \in \calJ \text{ and } x_{j}(I) = 1\}$. Then also $|1_{\calJ}(I_{1}) - 1_{\calJ}(I_{2})| \leq 1$ for adjacent $I_{1},I_{2} \in \calD_{n}$ by the previous argument, and 
\begin{displaymath} \Theta(I) = (1 + 2\alpha)^{1_{\calJ}(I)}(1 - \alpha)^{|\calJ_{n}| - 1_{\calJ}(I)}, \qquad I \in \calD_{n}, \: n \geq 1. \end{displaymath}
The inequality \eqref{form49} follows. \end{proof}

 The doubling constant of $\mu$ will be denoted by $D = D_{\alpha} = D_{\epsilon,s} \geq 1$:
\begin{displaymath} \mu(I_{1}) \leq D\mu(I_{2}), \end{displaymath}
whenever $I_{1},I_{2} \subset \R$ are adjacent intervals of the same length. If $\epsilon$ is small, the constant $\alpha$ needs to be chosen close to one, which increases the doubling constant $D$.
We also note that there exists a constant $\tau = \tau_{\alpha} < 1$ with the following property: If $I$ is a ternary interval and $J \subset I$ is one of the ternary children of $J$, then $\mu(J) \leq \tau \cdot \mu(I)$.
Inductively, we obtain the following:  If $J\subset I$ are ternary intervals of length $3^{-j}$ and $3^{-i}$ respectively then
\begin{equation}\label{tau}\mu(J) \leq \tau^{(j-i)}\mu(I).\end{equation}

For the eventual construction of $F$ to work, we require three fairly abstract properties from $\mu$, listed below as \nref{G0}-\nref{G2}. In the remainder of the section, we will verify that the properties are satisfied, if the index set $\calJ$ is chosen appropriately.
\begin{enumerate}
\item[(G0) \phantomsection\label{G0}] The measure $\mu$ resembles Lebesgue measure for all large scales, where the definition of "large" depends on $\alpha$ (which in turn only depends on $\epsilon$ and $s$). More precisely, for a suitable integer $n_{\alpha} \in \N$ to be determined later, the following holds: if $I \subset \R$ is any interval of length $|I| \geq r_{n_{\alpha}}$, then $|I|/3 \leq \mu(I) \leq 2|I|$.
Moreover, if $I$ is, in addition, a ternary interval, then $\mu(I)=|I|$.
\item[(G1) \phantomsection\label{G1}] If $I \in \calS_{n - 1}$, $n \geq 1$, and $J \subset 3I$ is any interval of length $s_{n} \leq |J| \leq r_{n - 1}$, then 
\begin{displaymath} \Theta(J) \sim_{D} \Theta(I), \end{displaymath}
where $\Theta(P) = \mu(P)/\ell(P)$.
Moreover, we have $\Theta(J)=\Theta(I)$ if, in addition, $J\subset I$ is a ternary interval.
\item[(G2) \phantomsection\label{G2}] For a ternary interval $I \in \calD_{n}$, let $I^{r}$ be the \emph{right neighbour} of $I$: that is, $I_{r} \in \calD_{n}$ is the interval immediately to the right from $I$. Define the coefficient
\begin{displaymath} a_{I} := \frac{\mu(I^{r}) - \mu(I)}{\mu(I)}. \end{displaymath}
The numbers $(a_{I})_{I \in \calS}$ form an \emph{$\calS$-Carleson sequence}. This means that
\begin{equation}\label{sparseCarleson} \mathop{\sum_{I \in \calS}}_{I \subset J} |a_{I}|\mu(I) = \mathop{\sum_{I \in \calS}}_{I \subset J} |\mu(I^{r}) - \mu(I)| \leq C_{\infty}\mu(J), \qquad J \in \calS, \end{equation}
where $C_{\infty} \geq 1$ is a constant depending only on $D$.
Note that, by the doubling condition, we immediately have $|a_{I}|\leq D$.
\end{enumerate}

Heuristically, condition \nref{G1} requires that, for scales $r$ with $s_{n} \leq r \leq r_{n - 1}$, the $\mu$-density of intervals of length $r$ is approximately determined by the density of a "parent" of length $r_{n - 1}$ (the densities among the parents can differ significantly, however). Note that this places no restrictions on how the density of $\mu$ varies on scales $r$ with $r_{n} < r < s_{n}$. Since $r_{n}$ is significantly smaller than $s_{n}$ by the rapid decay of the sequence $(r_{n})$, and the assumption $s < 1$, this leaves plenty of freedom to make $\mu$ into a highly singular measure.

\begin{remark}\label{aInftyRemark} Condition \nref{G2} is the main reason, why our construction does not work for Ahlfors-David regular sets $E$. We need the Carleson condition \eqref{sparseCarleson} for all scales $\calS$ appearing in the construction of $E$. So, if $E$ were Ahlfors-David regular, then $\calS$ would essentially have to contain all the scales smaller than one. However, according to a result of Buckley, Theorem 2.2(iii) in \cite{Bu}, if $\mu$ is doubling, then having the Carleson condition \eqref{sparseCarleson} for all scales implies $\mu \in A_{\infty}$. We need $\mu$ to be highly singular for the purposes of dimension-distortion, so a measure $\mu \in A_{\infty}$ cannot work for us. \end{remark}

We start describing the requirements on the indices $\calJ$. Initially, let $\calJ := \N$. We will next delete three subsets of $\calJ$. The first deletion is simple: we remove from $\calJ$ all the indices 
\begin{displaymath} j \in \{0,\ldots,2^{n_{\alpha} + m_{s}}\}, \end{displaymath}
where $n_{\alpha} \in \N$ is the number from \nref{G0}. Recall that $r_{n_{\alpha}} = 3^{-2^{n_{\alpha} + m_{s}}}$. Now, if $I$ is any ternary interval of length at least $r_{n_{\alpha}}/3$, then
\begin{displaymath} \mu(I) = \pi_{\alpha,\mathcal{J}}(I)|I| = |I|, \end{displaymath}
since the product in the definition of $\pi_{\alpha,\mathcal{J}}(I)$ is empty. It follows that if $I$ is an arbitrary interval of length $|I| \geq r_{n_{\alpha}}$, then $|I|/3 \leq \mu(I) \leq 2|I|$. This gives \nref{G0}.

 The second deletion is specified by the following requirement:
\begin{equation}\label{form34} n \geq 1 \text{ and } s_{n} < 3^{-j} \leq r_{n - 1} \quad \Longrightarrow \quad j \notin \calJ. \end{equation}
\begin{lemma} If $\mu$ has the form \eqref{form33}, and the collection $\calJ$ satisfies the requirement \eqref{form34}, then $\mu$ has property \textup{\nref{G1}}.
\end{lemma}

\begin{proof} Since $\mu$ is $D$-doubling, and the implicit constants in \nref{G1} are allowed to depend on $D$, it suffices to verify the "moreover" statement: if $I \in \calS_{n - 1}$, $n \geq 1$, and $J \subset I$ is a ternary subinterval with $s_{n} \leq |J| \leq r_{n - 1}$, then $\Theta(J) = \Theta(I)$. This follows immediately from the density formula \eqref{densityFormula}, observing (via \eqref{form34}) that the products defining $\pi(I)$ and $\pi(J)$ contain precisely the same indices of $\calJ$. \end{proof}

We turn to the final property \nref{G2}, and delete a third subset from $\calJ$:
\begin{equation}\label{form38} n \geq 1 \text{ and } r_{n} < 3^{-j} \leq 3^{n - 1}r_{n}  \quad \Longrightarrow \quad j \notin \calJ. \end{equation}
There will be no further deletions from $\calJ$, so the definition of $\calJ$ is now complete. The indices remaining in $\calJ$ can expressed as follows:
\begin{equation}\label{defCalJ} j \in \calJ \quad \Longleftrightarrow \quad j > 2^{n_{\alpha} + m_{s}} \text{ and } \quad 3^{n - 1}r_{n} < 3^{-j} \leq s_{n} \end{equation}
for some $n \geq 1$. These ranges are always non-empty by \eqref{form50}.

The usefulness of the deletion \eqref{form38} is explained by the following lemma:
\begin{lemma} Assume that $J \in \calS_{n}$, and $j \notin \calJ$ for $r_{n + 1} < 3^{-j} \leq 3^{A}r_{n + 1}$, where $A \in \N$. Assume also that $3^{A}r_{n + 1} \leq r_{n}$. Then
\begin{displaymath} \mathop{\sum_{I \in \calS_{n + 1}}}_{I \subset J} |\mu(I^{r}) - \mu(I)| = \mathop{\sum_{I \in \calS_{n + 1}}}_{I \subset J} |a_{I}|\mu(I) \leq \left(\frac{2}{3^{A}} + (D + 1)\tau^{A} \right) \cdot \mu(J). \end{displaymath}
\end{lemma}

\begin{proof} Consider $I \in \calS_{n + 1}$ with $I \subset J$. Regardless of the choice of $\calJ$, the expression
\begin{displaymath} \Theta_{n}^{n + 1}(x) := \mathop{\prod_{j \in \calJ}}_{r_{n + 1} < 3^{-j} \leq r_{n}} (1 + \alpha \cdot h(3^{j}x)) \end{displaymath}
is always constant on ternary intervals of length $r_{n + 1}$. So, we may write $\Theta_{n}^{n + 1}(x) := \Theta_{n}^{n + 1}(I)$ for $I \in \calS_{n}$. Then
\begin{equation}\label{form35} \mu(I) = \Theta_{n}^{n + 1}(I)\frac{|I|}{|J|} \cdot \mu(J), \qquad \text{for } I \in \calS_{n + 1} \text{ and } I \subset J. \end{equation}
Now, if $j \notin \calJ$ for $r_{n + 1} < 3^{-j} \leq 3^{A}r_{n + 1}$, then in fact $\Theta_{n}^{n + 1}(x)$ is constant on ternary intervals of length $3^{A}r_{n + 1}$. For $I \in \calS_{n + 1}$, let $\hat{I}$ be the unique ternary interval of length $3^{A}r_{n + 1}$ containing $I$; as before, we will write $\Theta_{n}^{n + 1}(\hat{I})$ for the value of $\Theta_{n}^{n + 1}(x)$ for $x \in \hat{I}$. There are two kinds of intervals in $I \in \calS_{n + 1}$ with $I \subset J$: those with $\hat{I} = \widehat{I^{r}}$, and those with $\hat{I} \neq \widehat{I^{r}}$. For $I$ of the first kind, \eqref{form35} shows that
\begin{displaymath} a_{I}\mu(I) = \mu(I^{r}) - \mu(I) = 0, \end{displaymath}
and consequently 
\begin{equation}\label{form36} \mathop{\sum_{I \in \calS_{n + 1}}}_{\hat{I} = \widehat{I^{r}} \subset J} |a_{I}|\mu(I) = 0. \end{equation}
For the intervals of the second kind, there are still two different possibilities: either $\widehat{I^{r}} \subset J$, or $\widehat{I^{r}} \not\subset J$. There is exactly one interval $I \in \calS_{n + 1}$ with $I \subset J$ of the latter kind, and we deal with it later. For now, assume that $\hat{I},\widehat{I^{r}} \subset J$. Then
\begin{displaymath} \mu(\hat{I}) = \Theta_{n}^{n + 1}(\hat{I})\frac{|\hat{I}|}{|J|} \cdot \mu(J) \quad \text{and} \quad \mu(\widehat{I^{r}}) = \Theta_{n}^{n + 1}(\widehat{I^{r}})\frac{|\widehat{I^{r}}|}{|J|} \cdot \mu(J), \end{displaymath}
so that
\begin{align*} |a_{I}|\mu(I) = |\mu(I^{r}) - \mu(I)| & = \left| \Theta_{n}^{n + 1}(\hat{I})\frac{|\hat{I}|}{3^{A}|J|} \cdot \mu(J) - \Theta_{n}^{n + 1}(\widehat{I^{r}})\frac{|\widehat{I^{r}}|}{3^{A}|J|} \cdot \mu(J) \right|\\
& = \frac{|\mu(\hat{I}) - \mu(\widehat{I^{r}})|}{3^{A}} \leq \frac{[\mu(\hat{I}) + \mu(\widehat{I^{r}})]}{3^{A}}.  \end{align*} 
It follows that
\begin{equation}\label{form37} \mathop{\sum_{I \in \calS_{n + 1}}}_{\hat{I} \neq \widehat{I^{r}} \subset J} |a_{I}|\mu(I) \leq \frac{2 \cdot \mu(J)}{3^{A}}, \end{equation}
noting that every $\hat{I}$ (and $\widehat{I^{r}}$) in the sum above corresponds to exactly one interval $I \in \calS_{n + 1}$ (namely the rightmost subinterval of $\hat{I}$). So, viewing \eqref{form36} and \eqref{form37}, the only remaining task is to treat the term $|a_{I}|\mu(I)$ for the rightmost interval $I \subset J$ with $|I| = r_{n + 1}$ (that is, the unique interval $I$ with $\widehat{I^{r}} \not\subset J$). Since $3^{A}r_{n + 1} \leq r_{n} = |J|$, we have
\begin{displaymath} \mu(I) \leq \tau^{A}\mu(J) \quad \text{and} \quad \mu(I^{r}) \leq D\mu(I) \leq D\tau^{A}\mu(J). \end{displaymath}
Consequently $|a_{I}|\mu(I) \leq \tau^{A}(D + 1)\mu(J)$, which proves the lemma.
\end{proof}

Now, we can easily verify the property \nref{G2}:
\begin{lemma} If \eqref{form38} is satisfied, then $\mu$ satisfies the $\calS$-Carleson condition
\begin{displaymath} \mathop{\sum_{I \in \calS}}_{I \subset J} |a_{I}|\mu(I) \lesssim_{\alpha} \mu(J), \qquad J \in \calS. \end{displaymath}
\end{lemma}

\begin{proof} Let $\delta_{k} := (2/3^{k} + (D + 1)\tau^{k})$ be the constant (with $A = k \geq 1$) from the previous lemma. Consider an interval $J \in \calS_{n}$. We verify by induction on $m \geq 1$ that
\begin{displaymath}\mathop{\mathop{\sum_{I \in \calS}}_{I \subset J}}_{|I| \geq r_{n + m}} |a_{I}|\mu(I) \leq \left(D + \sum_{k = 1}^{m} \delta_{n + k - 1} \right)\mu(J). \end{displaymath}
This is a direct consequence of the previous lemma. First, for $m = 1$, note that
\begin{displaymath} \mathop{\mathop{\sum_{I \in \calS}}_{I \subset J}}_{|I| \geq r_{n + 1}} |a_{I}|\mu(I) = |a_{J}|\mu(J) + \mathop{\sum_{I \in \calS_{n + 1}}}_{I \subset J} |a_{I}|\mu(I) \leq (D + \delta_{n})\mu(J). \end{displaymath}
Next, if the claim has already been verified for some $m \geq 1$, the case $m + 1$ is proven as follows:
\begin{align*} \mathop{\mathop{\sum_{I \in \calS}}_{I \subset J}}_{|I| \geq r_{n + m + 1}} |a_{I}|\mu(I) & = \mathop{\mathop{\sum_{I \in \calS}}_{I \subset J}}_{|I| \geq r_{n + m}} |a_{I}|\mu(I) + \mathop{\sum_{I \in \calS_{n + m}}}_{I \subset J} \mathop{\sum_{I' \in \calS_{n + m + 1}}}_{I' \subset I} |a_{I'}|\mu(I')\\
& \leq \left(D + \sum_{k = 1}^{m}\delta_{n + k - 1} \right)\mu(J) + \mathop{\sum_{I \in \calS_{n + m}}}_{I \subset J} \delta_{n + m} \cdot \mu(I)\\
& = \left(D + \sum_{k = 1}^{m + 1} \delta_{n + k - 1} \right)\mu(J). \end{align*}
The lemma follows by letting $m \to \infty$, noting that the geometric series converges. 
\end{proof}

We have now shown that the Riesz product $\mu$ associated to the index set $\calJ$, defined in \eqref{defCalJ}, satisfies the good properties \nref{G0}-\nref{G2}, with implicit constants depending only on $\alpha$ (and hence $\epsilon$ and $s$).

\subsection{Construction of the map}
 Recalling the construction of the measure $\mu$ from the previous section, and in particular observing that $\mu$ is a doubling measure with $\mu([0,1]) = 1$, we set
\begin{displaymath} f(t) := \mu([0,t]). \end{displaymath}
It is well-known that this defines a quasisymmetric homeomorphism $f \colon [0,1] \to [0,1]$, and it turns out that $\Hd f_{\sharp}\calL < \epsilon$, if the parameter $\alpha$ was chosen close enough to $1$ in the previous section, depending on $\epsilon$ and $s$. In this section, we construct another map $\g \colon [0,1] \times E \to \R$, which is "conjugate" to $f$, in the sense that the map $F(x_{1},x_{2},\ldots,x_{d})=(f(x_{1}),\g(x_{1}, x_{2}),\ldots,\g(x_{1},x_{d}))$ is a quasisymmetric embedding of $[0,1] \times E^{d - 1}$ to $\R^{d}$. A key property of $\g$ will be the following: for any $x,x'\in [0,1]$ and $y\in E$,
\begin{equation}\label{lipschitzProp} |\g(x,y) - \g(x',y)| \leq C_{\infty}|f(x) - f(x')|, \end{equation}
where $C_{\infty}$ is the constant from the Carleson condition in \nref{G2}. This allows us to transfer the dimension distortion of $f$ rather effortlessly to that of $F$.

To construct $\g$, we now need to take a somewhat abstract detour. 
\subsubsection{Carleson series} Let $\calS \subset \calD$ be a \emph{levelled} collection of ternary intervals, all contained in $[0,1)$ (in this section, "ternary" plays no particular role, but we stick to this terminology for consistency's sake). By \emph{levelled}, we mean that
\begin{displaymath} \calS = \bigcup_{k \geq 0} \calS_{k}, \end{displaymath}
where $\calS_{0} = \{[0,1)\}$, the families $\calS_{k}$ are pairwise disjoint, and each interval $I \in \calS_{k}$ is partitioned by finitely many intervals of $\calS_{k + 1}$; the theory will eventually be applied to the intervals $\calS_{n}$ introduced in \eqref{calSn}. If $I \in \calS_{k}$, then $k$ is the \emph{generation} of $I$, denoted by $k = \gen(I)$. Given a levelled collection of dyadic intervals and a probability measure $\mu$ on $[0,1)$, a sequence of complex numbers $(a_{I})_{I \in \calS}$ is called an \emph{$(\calS,\mu)$-Carleson sequence}, or simply \emph{$\calS$-Carleson sequence}, if there exists a constant $C \geq 1$ such that
\begin{equation}\label{carlesonAss} \mathop{\sum_{I \in \calS}}_{I \subset J} |a_{I}|\mu(I) \leq C\mu(J), \qquad J \in \calS. \end{equation}
This is precisely an abstract version of condition \nref{G2}. 
\begin{proposition}[Carleson series]\label{carlesonSeries} Let $\mu$ be a Borel probability measure with $\mu([0,1)) = 1$, and let $(a_{I})_{I \in \calS}$ be an $\calS$-Carleson sequence with constant $C \geq 1$, where $\calS \subset \calD$ is levelled. Then, there exist functions $\Delta_{k} \in L^{\infty}(\mu)$ with the following properties:
\begin{enumerate}
\item[\textup{(I)}\phantomsection\label{I}] For $I \in \calS$,
\begin{displaymath} \int_{I} \Delta_{\gen(I)} \, d\mu = a_{I}\mu(I). \end{displaymath}
\item[\textup{(II)}\phantomsection\label{II}] The function $\Delta_{\gen(I)}$ has the same sign (or "direction") as $a_{I}$ on $I \in \calS$. In other words, $a_{I} = 0$ implies $\Delta_{\gen(I)} \equiv 0$ on $I$, and
\begin{displaymath} \bar{a}_{I}\Delta_{I}(t) \geq 0 \text{ for $\mu$ a.e. } t \in I. \end{displaymath}
\item[\textup{(III)}\phantomsection\label{III}] The series $\sum \Delta_{k}$ converges absolutely $\mu$ almost everywhere, and moreover
\begin{displaymath} \sum_{k = 0}^{\infty} |\Delta_{k}| \leq C \quad \mu \text{ almost everywhere.} \end{displaymath}
\end{enumerate}
The sequence $(\Delta_{k})_{k \geq 0}$ will be called the \emph{Carleson series} associated with the sequence $(a_{I})_{I \in \calS}$.
\end{proposition} 
\begin{remark} The terminology comes from the fact that sums of the functions $\Delta_{k}$ correspond to Carleson sums of the sequence $(a_{I})_{I \in \calS}$. Namely, if $J \in \calS$, then
\begin{equation}\label{form13} \int_{J} \left( \sum_{k = \gen(J)}^{\infty} \Delta_{k} \right) \, d\mu = \mathop{\sum_{I \subset J}}_{I \in \calS} a_{I}\mu(I), \end{equation}
which follows immediately from \nref{I}. From \nref{II}, it follows that
\begin{equation}\label{form12} \int_{I} |\Delta_{\gen(I)}| \, d\mu = \frac{\bar{a}_{I}}{|a_{I}|} \int_{I} \Delta_{\gen(I)} \, d\mu = \frac{\bar{a}_{I}}{|a_{I}|} \cdot a_{I}\mu(I) = |a_{I}|\mu(I). \end{equation}
\end{remark}
\begin{proof}[Proof of Proposition \ref{carlesonSeries}] We start by constructing the finite Carleson series 
\begin{displaymath} (\Delta_{k}^{N})_{0 \leq k \leq N} = (\Delta_{k})_{0 \leq k \leq N} \end{displaymath}
associated with $(a_{I})_{I \in \calS_{N}}$, where $\calS_{N}$ is the following truncation of $\calS$:
\begin{displaymath} \calS_{N} = \bigcup_{k = 0}^{N} \calS_{k}. \end{displaymath}
The existence of the full series will eventually follow by an application of the Banach-Alaoglu theorem. 

The construction starts at the largest level $k = N$ and proceeds by induction towards the level $k = 0$. So, fix $I \in \calS_{N}$, and define
\begin{displaymath} \Delta_{N}^{N} := \Delta_{N}(t) = a_{I}, \qquad t \in I. \end{displaymath}
It follows from the Carleson assumption \eqref{carlesonAss} that $|\Delta_{N}| \leq C$. Since the intervals in $\calS_{N}$ partition $[0,1)$, the formula above defines $\Delta_{N}$ for all $t \in [0,1)$. It is clear that \nref{I}-\nref{II} hold. 

Next, assume that $\Delta_{N},\ldots,\Delta_{k + 1}$ have already been constructed for some $0 \leq k < N$ such that \nref{I}-\nref{II} hold for all $I \in \calS_{j}$ with $k + 1 \leq j \leq N$, and
\begin{equation}\label{form15} \sum_{j = k + 1}^{N} |\Delta_{j}(t)| \leq C, \qquad t \in [0,1). \end{equation}
Fix $J \in \calS_{k}$: the current plan is to define $\Delta_{k}$ on $J$. If $a_{J} = 0$, define $\Delta_{k} \equiv 0$ on $J$. So, assume that $a_{J} \neq 0$ in the sequel. Let $I_{1},\ldots,I_{n} \in \calS_{k + 1}$ be a partition of $I$. Then, by \eqref{form13}, and \eqref{form12} applied to $|\Delta_{k}|$ instead of $\Delta_{k}$, one has
\begin{align*} |a_{J}|\mu(J) + \int_{J} \left(\sum_{j = k + 1}^{N} |\Delta_{j}| \right) \, d\mu & = |a_{J}|\mu(J) + \sum_{i = 1}^{n} \int_{I_{i}} \left( \sum_{j = k + 1}^{N} |\Delta_{j}| \right) \, d\mu\\
& = |a_{J}|\mu(J) + \sum_{i = 1}^{n} \mathop{\sum_{I \in \calS_{N}}}_{I \subset I_{i}} |a_{I}|\mu(I)\\
& = \mathop{\sum_{I \in \calS_{N}}}_{I \subset J} |a_{I}|\mu(I) \leq C\mu(J). \end{align*} 
Consequently,
\begin{displaymath} \int_{J} \left( C - \sum_{j = k + 1}^{N} |\Delta_{j}| \right) \, d\mu \geq |a_{J}|\mu(J), \end{displaymath}
and one may pick a constant $\tau_{J} \in (0,1]$ such that
\begin{equation}\label{form14} \int_{J} \tau_{J} \left(C - \sum_{j = k + 1}^{N} |\Delta_{j}| \right) \, d\mu = |a_{J}|\mu(J). \end{equation}
Note that the integrand in \eqref{form14} is non-negative by \eqref{form15}. Now, set
\begin{displaymath} \Delta_{k}(t) := \frac{a_{J}\tau_{J}}{|a_{J}|} \left( C - \sum_{j = k + 1}^{N} |\Delta_{j}(t)| \right), \qquad t \in J. \end{displaymath}
Then $\Delta_{k}$ has the same sign as $a_{J}$ on $J$ in the sense of \nref{II}, and \nref{I} holds by \eqref{form14}. Finally, also \eqref{form15} holds with "$k + 1$" replaced by "$k$":
\begin{align*} \sum_{j = k}^{N} |\Delta_{j}(t)| & = \sum_{j = k + 1}^{N} |\Delta_{j}(t)| + \tau_{J} \left(C - \sum_{j = k + 1}^{N} |\Delta_{j}(t)| \right)\\
& = \left(\sum_{j = k + 1}^{N} |\Delta_{j}(t)| \right)(1 - \tau_{J}) + \tau_{J} \cdot C \leq C \cdot (1 - \tau_{J}) + \tau_{J} \cdot C = C. \end{align*}
This completes the definition of $\Delta_{k}$ on $J$. Since the intervals $J \in \calS_{k}$ partition $[0,1)$, this also completes the inductive definition of $\Delta_{k}$.

It remains to define the full Carleson series $(\Delta_{k})_{k \geq 0}$ associated with the sequence $(a_{I})_{I \in \calS}$. To this end, define the partial Carleson series $(\Delta_{k}^{N})_{0 \leq k \leq N}$, as above, for each $N \geq 0$. Since the sequence $(\Delta_{0}^{N})_{N \geq 0}$ is uniformly bounded (by $C$) in $L^{\infty}(\mu)$, which is the dual of $L^{1}(\mu)$, the Banach-Alaoglu theorem implies that there is a subsequence $\{0(j)\}_{j \in \N} \subset \N$ such that $(\Delta_{0}^{N_{0(j)}})_{j \geq 0}$ converges in the weak*-topology of $L^{\infty}(\mu)$ to a function $\Delta_{0} \in L^{\infty}(\mu)$:
\begin{displaymath} \int \Delta_{0}^{N_{0(j)}} \cdot g \, d\mu \to \int \Delta_{0} \cdot g \, d\mu, \qquad g \in L^{1}(\mu). \end{displaymath}
Next, by the same argument, the sequence $\{0(j)\}_{j \geq 0}$ has a further subsequence $\{1(j)\}_{j \geq 0}$ such that $(\Delta_{1}^{N_{1(j)}})_{j \geq 0}$ converges to a function $\Delta_{1} \in L^{\infty}(\mu)$ in the weak*-topology of $L^{\infty}(\mu)$. When repeated \emph{ad infinitumn}, the process produces a sequence of functions $(\Delta_{k})_{k \geq 0}$, contained in the $C$-ball of $L^{\infty}(\mu)$. We claim that $(\Delta_{k})_{k \geq 0}$ is the Carleson series associated with $(a_{I})_{I \in \calS}$. 

Condition \nref{I} is clear from the definition of weak* convergence: if $I \in \calS_{k}$, then
\begin{displaymath} \int_{I} \Delta_{k} \, d\mu = \int \Delta_{k}\chi_{I} \, d\mu = \lim_{j \to \infty} \int \Delta_{k}^{N_{k(j)}}\chi_{I} \, d\mu = a_{I}\mu(I). \end{displaymath}
Condition \nref{II} can be proven similarly: the sign of $\Delta_{k}$ cannot differ from the (common) signs of the approximating functions $\Delta_{k}^{N_{k(j)}}$ on $I$ (unless $a_{I} = 0$, in which case $\Delta_{k}$, and all the approximating functions, vanish on $I$).

To prove \nref{III}, it suffices to show that
\begin{displaymath} \int \left( \sum_{k = 0}^{M} |\Delta_{k}| \right) \cdot g \, d\mu \leq C, \end{displaymath}
for all non-negative $g \in L^{1}(\mu)$ with $\|g\|_{L^{1}(\mu)} \leq 1$, and for any $M \geq 0$. First, observe that 
\begin{displaymath} \int |\Delta_{k}| \cdot g \, d\mu = \lim_{j \to \infty} \int |\Delta_{k}^{N_{k(j)}}| \cdot g \, d\mu = \lim_{j \to \infty} \int |\Delta_{k}^{N_{M(j)}}| \cdot g \, d\mu, \quad 0 \leq k \leq M, \end{displaymath}
where the first equation is just the convergence of the sequence $\Delta_{k}^{N_{k(j)}}$ to $\Delta_{k}$, and the second follows from the fact that $\{M(j)\}_{j \geq 0}$ is a subsequence of $\{k(j)\}_{j \geq 0}$ for every $0 \leq k \leq M$. Thus
\begin{displaymath} \int \left( \sum_{k = 0}^{M} |\Delta_{k}| \right) \cdot g \, d\mu = \lim_{j \to \infty} \int \left( \sum_{k = 0}^{M} |\Delta_{k}^{N_{M(j)}}| \right) \cdot g \, d\mu \leq \lim_{j \to \infty} \int \left(\sum_{k = 0}^{N_{M(j)}} |\Delta_{k}^{N_{M(j)}}| \right) \cdot g \, d\mu \leq C, \end{displaymath}
where the last inequality follows from \eqref{form15} with "$k + 1$" replaced by "$0$". The proof of Proposition \ref{carlesonSeries} is complete. 
\end{proof}

Now, we return to the "real world" and construct the map $\g \colon [0,1]\times E \to \R$. Recall the collections $(\calS_{n})_{n \geq 0}$ of ternary intervals with lengths $(r_{n})_{n \geq 0}$ introduced in \eqref{calSn}, and the coefficients $a_{I} = [\mu(I^{r}) - \mu(I)]/\mu(I)$ from condition \nref{G2}. As mentioned above Proposition \ref{carlesonSeries}, condition \nref{G2} stipulates that the numbers $(a_{I})_{I \in \calS}$ form an $(\calS,\mu)$-Carleson sequence with respect to the probability measure $\mu|_{[0,1)}$. So, by Proposition \ref{carlesonSeries}, there exists an associated Carleson series $(\Delta_{k})_{k \geq 0}$, with
\begin{displaymath} \sum_{k = 0}^{\infty} |\Delta_{k}(t)| \leq C_{\infty} \end{displaymath}
at $\mu$ almost every $t \in [0,1)$. Recall the definition of the level $k$ parents of $y\in E$, $y_{k}(y) \in Y_{k}$ from Definition \ref{parents}.
\begin{definition}[The map $g$]\label{g-definition}
  For $(x,y)\in [0,1]\times E$ define
\begin{displaymath} \g(0,y) := \sum_{k = 0}^{\infty} \Theta([0,r_{k}])(y_{k+1}(y) - y_{k}(y))\end{displaymath}
and
\begin{displaymath} \g(x,y) := \g(0,y) + \int_{0}^{x} \sum_{k = 0}^{\infty} \frac{y_{k+1}(y) - y_{k}(y)}{r_{k}} \Delta_{k}(t) \, d\mu(t). \end{displaymath}
\end{definition}

We conclude with some basic properties of $\g$, in particular the the key inequality \eqref{lipschitzProp}:

\begin{lemma}\label{linfinity-bound}
  The function $\g \colon [0,1]\times E \to \R$ is well defined and continuous.
  Moreover, for any $0 \leq x \leq x' \leq 1$ and $y \in E$,
  \begin{equation}\label{displayedInequality} |\g(x,y) - \g(x',y)| \leq C_{\infty}\mu([x,x']) = C_{\infty}(f(x') - f(x)). \end{equation}
\end{lemma}

\begin{proof}
  Let $(x,y)\in [0,1]\times E$.
  First note that $\mu([0,r_{k}])$ decays geometrically when $k \to \infty$ (by \eqref{tau} and because the $r_{k}$ decay at least geometrically), and
  \begin{equation}\label{eq:111}|y_{k+1}(y) - y_{k}(y)| \leq |I_{k}| = r_{k},\end{equation}
  where $I_{k} \in \calI_{k}$ is the unique interval containing both $y_{k+1}(y)$ and $y_{k}(y)$.
  Thus, the first series appearing in the definition of $\g$ converges.

  For the second series, by property \nref{III} of the Carleson series and \eqref{eq:111}, the integrand is bounded in absolute value by
\begin{equation}\label{form51} \sum_{k = 0}^{\infty} |\Delta_{k}(t)| \leq C_{\infty} \quad \text{for $\mu$ a.e. } t\in [0,1].\end{equation}
Thus $\g$ is well defined.

To prove \eqref{displayedInequality}, let $x'\in [0,1]$ with $x'\geq x$ and note that from the definition of $\g$, \eqref{eq:111} and \eqref{form51},
  \begin{equation}\label{eq:222}|\g(x, y) -\g(x',y)| = \left|\int_{x}^{x'} \sum_{k = 0}^{\infty} \frac{y_{k+1}(y) - y_{k}(y)}{r_{k}} \Delta_{k}(t) \, d\mu(t) \right| \leq C_{\infty} \mu([x,x']).\end{equation}
  Also, observe that the equation $\mu([x,x']) = f(x') - f(x)$ is immediate from the definition $f(t) = \mu([0,t])$. The claim follows.
  
Finally, to see that $g$ is continuous, let $y'\in E$.
Then by \eqref{eq:222}, \eqref{form51} and the triangle inequality,
\begin{align*}|\g(x,y)-\g(x',y')| &\leq |g(0,y)-g(0,y')| + C_{\infty}\mu([x,x'])\\
                                &\leq_{D, C_{\infty}} \mu([0,r_{n}]) + \mu([x,x'])
\end{align*}
for $n\in \N$ maximal with $y_{n}(y)=y_{n}(y')$ (if $y=y'$ then the first term vanishes).
In particular, if $(x,y')\to (x',y)$ then $n\to \infty$ and so $\mu([0,r_{n}]),\mu([x,x']) \to 0$.  The continuity of $\g$ follows. \end{proof}

\section{The relationship between $\g$ and $\mu$}
In this section, we find upper and lower bounds for $|\g(x) - \g(b)|$ in terms of the measure $\mu([x_{1},x_{1} + |x - b|))$, for $x,b \in [0,1] \times E$.  This will allow us to show that $F$ is a quasisymmetry and that it reduces the dimension of measures of the form $\calL \times \nu$ by about as much as $f$ reduces the dimension of $\calL$.

\begin{notation}\label{notationForTehcnicalLemmas}\label{technical-lemmas} For this section, we fix $x = (x_{1},x_{2}), b = (b_{1},b_{2}) \in [0,1) \times E$, with $x_{2} \neq b_{2}$. For each $k\in \N$ let 
\begin{displaymath} I_{k} := [d_{k}, d_{k}+r_{k}) \in \calS_{k} \end{displaymath}
be the unique interval containing $x_{1}$.
Finally, let $n\geq 0$ be the greatest integer for which $y_{n}(x_{2})=y_{n}(b_{2})$, which exists because $x_{2} \neq b_{2}$. In other words, $x_{2},b_{2}$ share a common interval in $\calI_{n}$, but lie in different intervals of $\calI_{n + 1}$. This implies that
\begin{equation}\label{form64} s_{n + 1} \leq |x_{2} - b_{2}| \leq r_{n}. \end{equation}
\end{notation}

In order to prove estimates for $|g(x) - g(b)|$, we start by observing that
\begin{displaymath} g(x) - g(b) = [g(x) - g(x_{1},b_{2})] + [g(x_{1},b_{2}) - g(b)]. \end{displaymath} 
In this subsection, we seek upper and lower bounds for the first term, $g(x) - g(x_{1},b_{2})$. Note that
\begin{displaymath} \g(x) - \g(x_{1},b_{2}) = \sum_{k = n}^{\infty} \frac{\delta_k}{r_{k}} \left( \mu([0,r_{k}))  + \int_{0}^{x_{1}} \Delta_{k}(t) \, d\mu(t) \right), \end{displaymath}
where $\delta_{k}=(y_{k+1}(x_{2}) - y_{k}(x_{2}))-(y_{k+1}(b_{2}) - y_{k}(b_{2}))$. The terms for $k < n$ simply cancel, recalling the definition of $g$ from Definition \ref{g-definition}, by the definition of $n$.  As for the remaining terms, the term with $k = n$ is the "main term", and the terms with $k > n$ are "errors", which are further divided into two categories: those with $n<k \leq n_{\alpha}$ (possibly none) and those given by $k> \max\{n, n_{\alpha}\}$. For the first category, we typically get -- and need -- better estimates. 

In our first lemma, we significantly use properties \nref{I}--\nref{II} of the Carleson series.
\begin{lemma}\label{use-of-carleson-II}
  For any $k\in \N$,
  \[ \min\{\mu(I_{k}), \mu(I_{k}^{r})\} \leq \left| \mu([0,r_{k})) + \int_{0}^{x_{1}} \Delta_{k}(t) \, d\mu(t) \right| \leq \begin{cases}\mu(I_{k}) & k \leq n_{\alpha}\\
    2D\mu(I_{k}) & \text{otherwise.}\end{cases}\]
\end{lemma}

\begin{proof}
For any $k\in \N$, the definition of $\Delta_{k}$, namely \nref{I}, gives
\begin{displaymath} \int_{I} \Delta_{k}(t) \, d\mu(t) = a_{I}\mu(I) := \mu(I^{r}) - \mu(I)
  \quad I \in \calS_{k} = \calD_{r_{k}}, \end{displaymath}
  and so one finds that 
  \begin{equation}\label{form23} \mu([0,r_{k})) + \int_{0}^{d} \Delta_{k}(t) \, d\mu(t) = \mu([d,d + r_{k}))\end{equation}
  whenever $d \in [0,1)$ is the left endpoint of an interval in $\calS_{k}$.

For the first inequality, we use \eqref{form23} with $d=d_{k}$, the left endpoint of $I_{k} \ni x_{1}$, to get
\[\left|\mu([0,r_{k})) +\int_{0}^{x_{1}}\Delta_{k}(t)\, d\mu(t)\right| = \left|\mu(I_{k}) +\int_{d}^{x_{1}}\Delta_{k}(t)\, d\mu(t)\right|.\]
If $a_{I_{k}} \geq 0$ then $\Delta_{k}(t)\geq 0$ for $\mu$ a.e.\ $t\in I_{k}$ by property \nref{II} and so the right hand side is bounded below by $\mu(I_{k})$.
On the other hand, if $a_{I_{k}}<0$, then by the same reasoning $\Delta_{k}(t) \leq 0$ on $I_{k}$, so the right hand side is smallest when $x_{1}= d_{k}+r_{k}$, and is consequently bounded below by $\mu(I_{k}) + a_{I_{k}}\mu(I_{k}) = \mu(I_{k}^{r})$ (using also \nref{I}).
In either case,
\[\left|\mu([0,r_{k})) +\int_{0}^{x_{1}}\Delta_{k}(t)\, d\mu(t)\right| \geq \min\{\mu(I_{k}), \mu(I_{k}^{r})\},\]
as required.

For the second inequality, we again use \eqref{form23} with $d_{k}$ to get
  \begin{align} & \left| \mu([0,r_{k})) + \int_{0}^{x_{1}} \Delta_{k}(t) \, d\mu(t) \right|\notag\\
  &\quad \leq \left| \mu([0,r_{k})) + \int_{0}^{d_{k}} \Delta_{k}(t) \, d\mu(t) \right| + \int_{I_{k}} |\Delta_{k}(t)| \, d\mu(t)\notag\\
  &\quad = \mu(I_{k}) + |a_{I_{k}}|\mu(I_{k}).\notag \end{align}
If $k \leq n_{\alpha}$, then $a_{I_{k}}\mu(I_{k}) = \mu(I_{k}^{r}) - \mu(I_{k}) = |I_{k}^{r}| - |I_{k}^{r}| = 0$ by \nref{G0}. Therefore
\[\left| \mu([0,r_{k})) + \int_{0}^{x_{1}} \Delta_{k}(t) \, d\mu(t) \right| \leq \begin{cases}\mu(I_{k}) & k \leq n_{\alpha}\\
    2D\mu(I_{k}) & \text{otherwise,}\end{cases}\]
establishing the second inequality. \end{proof}

The next lemma gives a further estimate for the numbers appearing on the right hand side of the bound in the previous lemma. This is the first place where the choice of $n_{\alpha}$ becomes significant: it needs to be so large that the errors corresponding to $k > \max\{n,n_{\alpha}\}$ admit an upper bound depending only on $n$, not $D$, as in the first estimate below:

\begin{lemma}\label{G0-control}
  For any $D \geq 1$, if $n_{\alpha}$ is sufficiently large then
  \[\sum_{n_{\alpha} \geq k > n}2\mu(I_{k}) + \sum_{k>n_{\alpha}} 4D\mu(I_{k})\leq \frac{3}{4}s_{n+1}\]
  whenever $0 \leq n < n_{\alpha}$.
  Moreover, there exists a $D'$ depending only upon $D$ such that
  \[\sum_{k > n}\mu(I_{k}) \leq D'\mu(I_{n+1})\]
  for any $n\in\N$.
\end{lemma}

\begin{proof}
  We prove the statements in reverse order.
  Recall the definition of $\tau<1$ (which depends only upon $D$) from \eqref{tau} to obtain
  \[\mu(I_{k+1}) \leq \tau^{\log_{3} (r_{k}/r_{k+1})} \mu(I_{k}) = \left(\frac{r_{k+1}}{r_{k}}\right)^{D'} \mu(I_{k}) \leq 3^{-kD'} \mu(I_{k}),\]
  for some $D'>0$ depending only upon $D$.
  For the final inequality, we used \eqref{form50}.
  Notice that the first inequality implies that the $\mu(I_{k})$ decay at least at a geometric rate, and so we already obtain the final statement of the lemma.
  Moreover,
  \begin{equation}\label{77} \sum_{k>n} \mu(I_{k}) \leq D' \mu(I_{n+1})\leq D' 3^{-nD'} \mu(I_{n}).\end{equation}

  To obtain the first statement, suppose $n_{\alpha}$ is sufficiently large so that $DD' 3^{-D'n_{\alpha}}<1/29$.
  Then, for any $n < n_{\alpha}$, by
  property \nref{G0} and \eqref{77},
  \begin{align*} \sum_{n_{\alpha}\geq k > n}2\mu(I_{k}) + \sum_{k> n_{\alpha}}4D\mu(I_{k})
                                     &\leq \sum_{n_{\alpha} \geq k >n} 2r_{k} + 4DD' 3^{-D'n_{\alpha}}\mu(I_{n_{\alpha}})\\
    &\leq 60r_{n+1}/29 +DD' 3^{-D'n_{\alpha}}r_{n+1}
  \end{align*}
  since the $r_{k+1}<r_{k}/30$ and hence decay at least at a geometric rate. By the choice of $n_{\alpha}$, this final expression is bounded above by $61r_{n+1}/29$.
  Recalling \eqref{form50} completes the proof since
  \[61r_{n+1}/29\leq 61\cdot 3^{-(n+1)}s_{n+1}/29 \leq 3s_{n+1}/4.\] 
  \end{proof}

We are now in a position to estimate the difference $|\g(x)-\g(x_{1},b_{2})|$:
\begin{lemma}\label{vertical-big}
The following inequality is true:
  \begin{equation}\label{form61}
    \left|\g(x) - \g(x_{1}, b_{2}) \right| \lesssim_{D} |y_{n+1}(x_{2})-y_{n+1}(b_{2})| \Theta(I_{n}) + \mu(I_{n+1}). \end{equation}
  Further, provided $n_{\alpha}$ is so large that the conclusion of Lemma \ref{G0-control} holds, then
\begin{equation}\label{form62} |\g(x)-\g(x_{1},b_{2})| \geq |y_{n+1}(x_{2})-y_{n+1}(b_{2})|\min\{\Theta(I_{n}),\Theta(I_{n}^{r})\} - \mathcal E(n+1),\end{equation}
  for
  \[\mathcal E(n+1)=\begin{cases}\frac{3}{4}s_{n+1} & n<n_{\alpha}\\ D''\mu(I_{n+1}) &n\geq n_{\alpha}\end{cases}\]  
  and $D''$ a constant depending only upon $D$.
\end{lemma}

\begin{proof}
Recall that
 \begin{equation}\label{eq1}\g(x) - \g(x_{1},b_{2}) = \sum_{k = n}^{\infty} \frac{\delta_k}{r_{k}} \left( \mu([0,r_{k}))  + \int_{0}^{x_{1}} \Delta_{k}(t) \, d\mu(t) \right), \end{equation}
 where $\delta_k =(y_{k+1}(x_{2}) - y_{k}(x_{2}))-(y_{k+1}(b_{2}) - y_{k}(b_{2}))$. To prove \eqref{form61}, we first use the second inequality of Lemma \ref{use-of-carleson-II} to obtain
\[|\g(x) - \g(x_{1}, b_{2})| \lesssim_{D} |\delta_{n}|\Theta(I_{n}) + \sum_{k > n} \frac{|\delta_{k}|}{r_{k}}\mu(I_{k}).\]
Since $\delta_{n}=y_{n+1}(x_{2})-y_{n+1}(b_{2})$, the first term is of the required form.
Further, notice that
\[\frac{|\delta_{k}|}{r_{k}}\leq \frac{|y_{k+1}(x_{2})-y_{k}(x_{2})|}{r_{k}} + \frac{|y_{k+1}(b_{2})-y_{k}(b_{2})|}{r_{k}} \leq 2, \qquad k \geq n. \]
Thus, an application of the second inequality from Lemma \ref{G0-control} gives \eqref{form61}.

The proof of \eqref{form62} is very similar, with the only difference  that we bound the $k=n$ term of \eqref{eq1} from below using the first inequality of Lemma \ref{use-of-carleson-II} and are more careful in treating the error terms.
Indeed, by applying Lemma \ref{use-of-carleson-II} to \eqref{eq1}, using the triangle inequality and the fact that $|\delta_{k}|/r_{k}\leq 2$, we get
\[|\g(x)-\g(x_{1},b_{2})| \geq |\delta_{n}| \cdot \min\{\Theta(I_{n}),\Theta(I_{n}^{r})\}-\sum_{n_{\alpha}\geq k >n}2\mu(I_{k}) -\sum_{k>\max\{n,n_{\alpha}\}}4D\mu(I_{k}).\]
There are now two cases to consider.
If $n<n_{\alpha}$ then we apply the first inequality of Lemma \ref{G0-control} to deduce that the sum of the terms with a negative sign is less than $3s_{n+1}/4$, as required.
If $n \geq n_{\alpha}$ then the second term of the right hand side equals 0 and we apply the second inequality of Lemma \ref{G0-control} to deduce \eqref{form62}.
\end{proof}

The previous results easily give an upper bound of $|\g(x)-\g(x_{2},b_{2})|$:
\begin{lemma}\label{upperbound1}
 We have
  \[|\g(x)-\g(x_1,b_{2})| \lesssim_{D} \mu([x_1, x_1+ |x-b|)).\]
\end{lemma}

\begin{proof}
Note that since $n\geq 0$ is maximal with $y_n(x_{2})= y_n(b_{2})$, the construction of $E$ implies that $s_{n + 1} \leq |x_{2}-b_{2}| \leq r_n$, and $|x_{2} - b_{2}| \sim |y_{n + 1}(x_{2}) - y_{n + 1}(b_{2})|$. 
  We apply Lemma \ref{vertical-big} to obtain
  \begin{align*} |\g(x) - \g(x_{1},b_{2})| 
    &\lesssim_{D} |y_{n+1}(x_{2}) - y_{n+1}(b_{2})|\Theta(I_n) +  \mu(I_{n+1})\\
    &\lesssim_{D} |x_{2} - b_{2}| \Theta(I_{n}) + \mu([x_{1},x_{1} + r_{n+1})), \end{align*}
by the doubling property of $\mu$.
  
  Since $|x_{2}-b_{2}|\geq s_{n + 1} \geq r_{n+1}$, the second term is bounded above by $\mu([x_{1},x_{1} + |x_{2}-b_{2}|)$.
  Since $s_{n + 1} \leq |x_{2}-b_{2}| \leq r_n$, recalling \eqref{form64}, we apply \nref{G1} to get $\Theta (I_{n}) \sim_{D} \Theta([x_{1},x_{1} + |x_{2}-b_{2}|))$.  Therefore,
  \begin{align*} |\g(x) - \g(x_{1},b_{2})| & \lesssim_{D} |x_{2} - b_{2}|\Theta([x_{1},x_{1} + |x_{2} - b_{2}|)) + \mu([x_{1},x_{1} + |x_{2} - b_{2}|))\\
  & \leq 2\mu([x_{1},x_{1} + |x - b|)), \end{align*}
  as required.  
\end{proof}

The previous lemma has the following corollary for the difference $|g(x) - g(b)|$:

\begin{lemma}\label{fcomparable} The following estimate holds:
   \[|\g(x) - \g(b)| \lesssim_{D,C_{\infty}} \mu([x_{1},x_{1}+|x-b|)).\]
\end{lemma}

\begin{proof} We apply Lemmas \ref{upperbound1} and \ref{linfinity-bound}:
  \begin{align}
   \label{form63} |\g(x)-\g(b)| &\leq |\g(x)-\g(x_1,b_2)| + |\g(x_1,b_2) - \g(b)| \\
                    &\lesssim_{D,C_{\infty}} \mu([x_{1}, x_{1}+|x-b|))+\mu([x_{1}, x_{1}+|x_{1}-b_{1}|)),\notag\\
                    &\leq 2\mu([x_{1},x_{1}+|x-b|)).\notag
  \end{align}
\end{proof}
\begin{remark} Note that the estimate above remains valid for $x,b \in [0,1] \times E$, by continuity ($\mu$ is clearly non-atomic), even though our standing assumption is $x,b \in [0,1) \times E$. Also, the estimate remains valid, if $x_{2} = b_{2}$; then the first term on the right hand side of \eqref{form63} simply vanishes, and the estimate follows from Lemma \ref{linfinity-bound}. The same remark also applies without change to the next lemma. \end{remark}

Finally, we prove a lower bound, which matches the upper bound from the previous lemma. This is only true, if $|x_{1} - b_{1}| \ll |x - b|$. Here, we again need $n_{\alpha}$ to be large: otherwise we could only prove the estimate for $|x - b|$ sufficiently small, depending on the doubling constant $D$ and $C_{\infty}$, and this would be fatal for eventually proving quasisymmetry on "large scales". 

\begin{lemma}\label{lemma-f2-growth-rate}
  Let $\rho_{D} \in (0,\tfrac{1}{2})$ be sufficiently small (depending only on $D,C_{\infty}$) and $n_{\alpha}$ sufficiently large (depending only upon $D$).  If $|x_1-b_1| \leq \rho_{D} |x-b|$, then
  \[|\g(x)-\g(b)| \gtrsim_{D} \mu([x_{1}, x_{1} + |x-b|)).\]
\end{lemma}  

\begin{proof}  The definition of $n$ implies (recall \eqref{form64}) that
\begin{equation}\label{eq4}s_{n+1} \leq |y_{n+1}(x_{2}) - y_{n+1}(b_{2})| \leq r_{n} \text{ and } |y_{n+1}(x_{2})-y_{n+1}(b_{2})| \sim |x_{2}-b_{2}|, \end{equation}
where the second estimate also uses the fact that $s_{n + 1}$ is larger than $r_{n + 1}$.
From Lemma \ref{vertical-big}, we see that
\begin{equation}\label{calE-choice} |\g(x) - \g(x_{1},b_{2})| \geq \Theta(I_{n})|y_{n+1}(x_{2}) - y_{n+1}(b_{2})| - \mathcal E(n+1).\end{equation}
There are two cases to consider.  First, if $n<n_{\alpha}$ then $\Theta(I_{n}) = 1$ by \nref{G0}, so by the definition of $\mathcal E$,
\[|\g(x) - \g(x_{1},b_{2})| \geq |y_{n+1}(x_{2}) - y_{n+1}(b_{2})| - 3s_{n+1}/4.\]
Since $|y_{n+1}(x_{2})-y_{n+1}(b_{2})|\geq s_{n+1} \geq r_{n_{\alpha}}$, we further have
\begin{equation}\label{eq7}|\g(x) - \g(x_{1},b_{2})| \geq |y_{n+1}(x_{2})-y_{n+1}(b_{2})|/4 \sim \mu([x_{1}, x_{1} +|x_{2}-b_{2}|)),\end{equation}
using \nref{G0} and \eqref{eq4}.

For the second case, suppose that $n\geq n_{\alpha}$ so that \eqref{calE-choice} becomes
\begin{equation}\label{eq8} |\g(x) - \g(x_{1},b_{2})| \geq \Theta(I_{n})|y_{n+1}(x_{2}) - y_{n+1}(b_{2})| - D''\mu(I_{n+1}).\end{equation}
Observe that since $r_{n} \geq |x_{2}-b_{2}| \geq s_{n + 1}$, the condition \nref{G1} implies
\[\Theta(I_{n}) \sim_{D} \Theta([x_{1}, x_{1}+|x_{2}-b_{2}|)).\]
Therefore, by \eqref{eq4},
\begin{equation}\label{eq9}\Theta(I_{n})|y_{n+1}(x_{2}) - y_{n+1}(b_{2})| \sim_{D} \mu([x_{1}, x_{1} + |x_{2}-b_{2}|)).\end{equation}
Since $I_{n+1}$ is an interval of length $r_{n+1}\leq 3^{-n} s_{n + 1} \leq 3^{-n}|x_{2}-b_{2}|$, the quantity in \eqref{eq9} (including the constant depending on $D$) is much larger than $D''\mu(I_{n+1})$ provided $n$ is sufficiently large.
Since $n\geq n_{\alpha}$, we can ensure this by choosing $n_{\alpha}$ sufficiently large.
Therefore, we combine \eqref{eq8} and \eqref{eq9} to see that
\begin{equation}\label{eq10}|\g(x) - \g(x_{1},b_{2})| \gtrsim_{D} \mu([x_{1},x_{1}+|x_{2}-b_{2}|))\end{equation}
provided $n_{\alpha}$ is sufficiently large (depending only upon $D$).

Finally, we use the main assumption that $|x_{1} - b_{1}| \leq \rho_{D} |x - b|$.
By Lemma \ref{linfinity-bound} and the definition of $\tau$ (which depends only upon $D$) given in \eqref{tau},
\[|\g (x_{1}, b_{2})-\g (b)| \leq_{D,C_{\infty}} \mu([x_{1},x_{1}+|x_{1}-b_{1}|)) \lesssim_{D} \tau^{-\log \rho_{D}} \mu([x_{1}, x_{1}+|x-b|)).\]
Since $\rho_{D}\leq 1/2$, we have $|x-b|\leq 2|x_{2}-b_{2}|$ and so
\[|\g (x_{1}, b_{2})-\g (b)| \lesssim_{D} \tau^{-\log \rho_{D}} \mu([x_{1}, x_{1}+|x_{2}-b_{2}|)).\]
Therefore, provided $\rho_{D}$ is sufficiently small (depending only upon $D,C_{\infty}$), we may combine the previous equation with \eqref{eq7} or \eqref{eq10} depending on the case and use the triangle inequality to obtain
\begin{align*}
  |\g(x)-\g(b)| &\geq |\g(x)-\g(x_{1},b_{2})|-|\g(x_{1},b_{2})-\g(b)|\\
                                    &\gtrsim_{D} \mu([x_{1}, x_{1}+|x_{2}-b_{2}|))-\mu([x_{1}, x_{1}+|x_{2}-b_{2}|))/2\\
  &= \mu([x_{1}, x_{1}+|x_{2}-b_{2}|))/2,
\end{align*}
as required.
\end{proof}

\section{Quasisymmetry and dimension distortion}\label{dimensionDistortion}

In this section, we prove the main result, Theorem \ref{mainTechnical}. Recall the maps $f \colon [0,1] \to [0,1]$ and $\g \colon [0,1] \times E \to \R$ constructed in Section \ref{constructions}. For $d \geq 2$, we define the map $F \colon K :=[0,1] \times E^{d - 1} \to \R^{d}$ by setting
\begin{displaymath} F(x_{1}, \ldots, x_{d}) = (f(x_{1}), \g(x_{1}, x_{2}), \g(x_{1}, x_{3}), \ldots, \g(x_{1}, x_{d})). \end{displaymath}
The following tasks remain:
\begin{itemize}
\item Verify that $F$ is a quasisymmetric embedding of $K$ to $\R^{d}$.
\item Find a subset $K_{\epsilon} \subset K$, which has simultaneously full measure with respect to any measure of the form $\calL \times \nu$, where $\nu$ is Radon and supported on $E^{d - 1}$, and which has the property that $\Hd F(K_{\epsilon}) < \epsilon$.
\end{itemize}
We quickly remind the reader, how the various parameters in the construction depend on each other. The numbers $s,\epsilon > 0$ are "given", and determine how non-doubling the measure $\mu$ needs to be: in other words, $\alpha$ needs to be chosen close to one, which increases the doubling constant $D$. To prove that $F$ is quasisymmetric with this $D$, we need the results from the previous section. In particular, the number $n_{\alpha} \in \N$ has to be chosen large enough, and the number $\rho_{D} > 0$ needs to be chosen small enough.

\subsection{Quasisymmetry}\label{quasisymmetry} We now prove that $F$ is a quasisymmetric embedding on $K$. We treat the ambient dimension "$d$" as an absolute constant: $\lesssim_{d}$ is shortened to $\lesssim$. Assume that $\rho_{D} < 1$ and $n_{\alpha} \in \N$ are chosen so that we can apply Lemma \ref{lemma-f2-growth-rate}. We start by proving that $F$ is a \emph{weak quasisymmetry}:
\begin{equation}\label{wQS} a,b,x \in K \text{ and } |a - x| \leq |b - x| \quad \Longrightarrow \quad |F(a) - F(x)| \lesssim_{\mu} |F(b) - F(x)|. \end{equation}
The notation $\lesssim_{\mu}$ is shorthand for $\lesssim_{D,C_{\infty}}$, where $C_{\infty}$ is the constant from \nref{G2}. There are two essentially different cases: either $|b - x|$ is comparable to $|b_{1} - x_{1}|$, or $|b_{1} - x_{1}|$ is significantly smaller than $|x - b|$.

Fix the three points $x,a,b \in K$, and assume for convenience that $a_{1} \leq x_{1} \leq b_{1}$ (this only influences, should we write $[a_{1},x_{1}]$ or $[x_{1},a_{1}]$). First, suppose that 
\begin{displaymath} |b_1-x_1| \geq \left(\frac{\rho_{D}}{2d}\right) |b-x| \gtrsim_{\mu} |a - x| \geq |a_{1} - x_{1}|. \end{displaymath}
Since $\mu$ is doubling, this implies that
\begin{displaymath} |f(x_{1}) - f(a_{1})| = \mu([a_{1},x_{1}]) \lesssim_{\mu} \mu([x_{1},b_{1}]). \end{displaymath}
Therefore, recalling the definition of $F$, and using Lemma \ref{fcomparable} for all $(x_{1},x_{k}),(a_{1},a_{k}) \in [0,1] \times E$, $2 \leq k \leq d$, yields
\begin{align} |F(a)-F(x)| & \leq |f(x_{1}) - f(a_{1})| + \sum_{k = 2}^{d} |\g(a_{1},a_{k}) - \g(x_{1},x_{k})|\notag\\
&\label{form55} \lesssim_{\mu} \mu([x_{1},b_{1}]) + d \cdot \mu([x_{1},x_{1} + |x - a|])\\
& \lesssim_{\mu} \mu([x_{1},b_{1}]) = |f(b_1)-f(x_1)| \leq |F(b)-F(x)|, \notag \end{align}
as claimed.

Second, suppose that $|b_1 - x_1| < (\rho_{D}/2d) |b-x| < |b - x|/2$ and so there exists $2 \leq k \leq d$ such that 
\begin{displaymath} \rho_{D}|(b_{1},b_k) -(x_{1},x_k)| \geq \left(\frac{\rho_{D}}{2d}\right)|x - b| > |x_{1} - b_{1}|. \end{displaymath}
Lemma \ref{lemma-f2-growth-rate} is, hence, applicable to $x'=(x_1,x_{k}) \in [0,1] \times E$ and $b'=(b_1,b_{k}) \in [0,1] \times E$, and the conclusion is that
\[|\g(x')-\g(b')| \gtrsim_{\mu} \mu([x_{1}, x_{1}+|b' -x'|]) \gtrsim_{\mu} \mu([x_{1}, x_{1}+|x-b|]),\]
where the last inequality follows from doubling. Then, by using Lemma \ref{fcomparable} again as on line \eqref{form55}, we obtain,
\begin{align*}
  |F(x)-F(a)| & \lesssim_{\mu} \mu([x_{1}, x_1+|x-a|]) \leq \mu([x_{1}, x_{1}+|x-b|))\\
  &\lesssim_{\mu} |\g(x') - \g(b')| \leq |F(x)-F(b)|.
\end{align*}
This concludes the proof of $F$ being weakly quasisymmetric.

Finally, to see that $F$ is injective and "properly" quasisymmetric, we fix $a,b,x \in K$ with $x \neq b$. Consider the line segment
\begin{displaymath} L_{x} := [0,1] \times \{(x_{2},\ldots,x_{d})\} \subset K, \end{displaymath}
which contains $x$. Pick a point $b' \in L_{x}$ with the property that $|x - b'|$ is as close to $|x - b|$ as possible. If $|x - b| < 1/2$, then one can choose $|x - b'| = |x - b|$, and in general $|x - b'| \leq |x - b| \leq \sqrt{d}|x - b'|$. Note that $0 < |x - b'| = |x_{1} - b_{1}'|$. Then, by the weak quasisymmetry, established above,
\begin{equation}\label{form56} |F(b) - F(x)| \gtrsim_{\mu} |F(b') - F(x)| \geq |f(b_{1}') - f(x_{1})| > 0, \end{equation}
which proves that $F$ is injective. 

To prove quasisymmetry, pick similarly a point $a' \in L_{x}$ with $|a' - x|$ as close to $|a - x|$ as possible, so that also $|a' - x| \leq |a - x| \leq \sqrt{d}|a' - x|$. We claim that
\begin{equation}\label{form57} |F(a) - F(x)| \sim_{\mu} |F(a') - F(x)| \quad \text{and} \quad |F(b) - F(x)| \sim_{\mu} |F(b') - F(x)|. \end{equation} 
By symmetry, it suffices to consider just $x,b,b'$. If $|x - b| < 1/2$, then $|x - b'| = |x - b|$, and the claim follows from two applications of the weak quasisymmetry implication \eqref{wQS}. Otherwise, if $|x - b| > 1/2$, then $|x_{1} - b_{1}'| \geq \sqrt{d}/100$, and it follows from \eqref{form56}, plus the quasisymmetry of $f$, that 
\begin{displaymath} 1 \gtrsim_{\mu} |F(b) - F(x)| \gtrsim_{\mu} |F(b') - F(x)| \geq |f(b_{1}') - f(x_{1})| \sim_{\mu} 1. \end{displaymath}
This proves \eqref{form57}.

Finally, note that the weak quasisymmetry of $F$ on the line $L_{x}$ implies "proper" quasisymmetry on $L_{x}$, since $L_{x}$ is connected, see Heinonen's book \cite[Theorem 10.19]{H}. Thus
\[\frac{|F(a)-F(x)|}{|F(b)-F(x)|} \sim_{\mu} \frac{|F(a')-F(x)|}{|F(b')-F(x)|} \leq \eta_{x}\left(\frac{|a' - x|}{|b' - x|} \right)\]
for some homeomorphism $\eta_{x} \colon [0,\infty) \to [0,\infty)$. Moreover, the weak quasisymmetry constants of $F$ on a fixed line $L_{x}$ do not depend on the choice of $x$, so also $\eta_{x}$ can be chosen independently of $x$. Since $|a' - x|/|b' - x| \sim |a - x|/|b - x|$, this proves that $F$ is quasisymmetric. 

\subsection{Dimension distortion} We now proceed with the task of showing that $\Hd F(K_{\epsilon}) < \epsilon$ for a certain subset $K_{\epsilon} \subset K$, which has simultaneously full $(\calL \times \nu$)-measure for all Radon measures $\nu$ supported on $E^{d - 1}$. Unsurprisingly, the set $K_{\epsilon}$ has the form $G \times E^{d - 1}$, where $G = G_{\epsilon} \subset [0,1]$ is a set of full Lebesgue measure on $[0,1]$.

Fix $M = M_{d,\epsilon,s} \in \N$ large, and for each $n\in \N$ let 
\begin{displaymath} \calS^{M}_{n} := \{I \in \calS_{n} : |f(I)| \leq r_{n}^{1+M}\}. \end{displaymath}
Write also and $G_n := G_{n}^{M} := \cup \{I : I \in \calS_{n}^{M}\}$.
\begin{lemma}\label{dimension-distortion}
Let
\[G := G_{M} := \liminf_{n \to \infty} G_{n} = \bigcup_{k = 0}^{\infty} \bigcap_{n = k}^{\infty} G_{n} \]
and $B := [0,1) \setminus G$. Then, if the parameter $\alpha$ from the construction of $\mu$ is chosen close enough to one, depending only on $s$ and $M$ (hence $\epsilon$), then the set $B$ satisfies $\calL(B) = 0$. \end{lemma}

\begin{proof} During the proof of the lemma, the reader should view $\calL$ as a probability measure, and to emphasise this, we write $\tn := \calL$. Let $I = [a,b) \in \calS_{n}$, and note that
\begin{equation}\label{form42} |f(I)| = f(b) - f(a) = \int_{a}^{b} \, d\mu = \mu(I) = \Theta_{n}(I)|I|, \end{equation}
where $\Theta_{n}(I)$ is the (common) value of
\begin{displaymath} \Theta_{n}(x) = \prod_{j \in \calJ_{n}} (1 + \alpha \cdot h(3^{j}x)) \end{displaymath}
on the interval $I$, and
\begin{displaymath} \calJ_{n} := \{j \in \calJ : 3^{-j} > r_{n}\}. \end{displaymath}
For $j \in \calJ_{n}$, let $A_{j} \subset [0,1)$ be the event
\begin{displaymath} A_{j} := \{x \in [0,1) : h(3^{j}x) = -1\} = \{x \in [0,1) : 1 + \alpha \cdot h(3^{j}x) = 1 - \alpha\}, \end{displaymath} 
and write
\begin{displaymath} S_{n}(x) := \sum_{j \in \calJ_{n}} \chi_{A_{j}}. \end{displaymath}
The events $A_{j}$ are clearly independent (whether or not $x \in A_{j}$ depends only on the $j^{th}$ decimal in the ternary expansion of $x = 0.x_{0}x_{1},\ldots$), and have probability $\tn\{A_{j}\} = 2/3$. So $S_{n}$ is a sum of $|\calJ_{n}|$ independent random variables with expectation
\begin{displaymath} \E[S_{n}] = \frac{2|\calJ_{n}|}{3}. \end{displaymath}
Moreover, the value of $S_{n}$ is constant on intervals $I \in \calS_{n}$ (we will denote this constant by by $S_{n}(I)$) and the value of $\Theta_{n}(I)$ is determined by 
\begin{equation}\label{form39} \Theta_{n}(I) = (1 - \alpha)^{S_{n}(I)}(1 + 2\alpha)^{|\calJ_{n}| - S_{n}(I)}. \end{equation}

Write $\sigma := (1 - s)/4$ (recall that $\Hd E = s < 1$), and consider the event
\begin{displaymath} B_{n} := \{S_{n} < (\tfrac{2}{3} - \sigma)|\calJ_{n}|\}, \end{displaymath}
which is a union of certain intervals in $\calS_{n}$, denoted by $\calB_{n}$. By Chernoff's inequality,
\begin{equation}\label{form44} \tn\{B_{n}\} \leq \exp(-2\sigma^{2}|\calJ_{n}|). \end{equation}
If $I \in \calG_{n} := \calS_{n} \setminus \calB_{n}$, then
\begin{displaymath} S_{n}(I) \geq (\tfrac{2}{3} - \sigma)|\calJ_{n}| \quad \text{and} \quad |\calJ_{n}| - S_{n}(I) \leq (\tfrac{1}{3} + \sigma)|\calJ_{n}|, \end{displaymath} 
The first inequality is just the definition of $I \in \calG_{n}$, and if second failed, then
\begin{displaymath} |\calJ_{n}| = S_{n}(I) + |\calJ_{n}| - S_{n}(I) > [(\tfrac{2}{3} - \sigma) + (\tfrac{1}{3} + \sigma)]|\calJ_{n}| = |\calJ_{n}|, \end{displaymath}
which is absurd. So, if $I \in \calG_{n}$, the density formula \eqref{form39} implies that
\begin{equation}\label{form41} \Theta_{n}(I) \leq (1 - \alpha)^{(\tfrac{2}{3} - \sigma)|\calJ_{n}|}(1 + 2\alpha)^{(\tfrac{1}{3} + \sigma)|\calJ_{n}|} = [(1 - \alpha)^{\tfrac{2}{3} - \sigma}(1 + 2\alpha)^{\tfrac{1}{3} + \sigma}]^{|\calJ_{n}|}. \end{equation}
We claim that if $\alpha$ is chosen close enough to one, then the right hand side is bounded by $r_{n}^{M}$ for $I \in \calG_{n}$, and all $n \in \N$ large enough. Recall the restrictions on $\calJ$ given in \eqref{defCalJ}. Then the index family $\calJ_{n}$ contains all the indices $j \in \N$ such that $j > 2^{n_{\alpha} + m_{s}}$ and
\begin{equation}\label{form52} 3^{n - 1}r_{n} < 3^{-j} < s_{n} \sim r_{n - 1}^{s}r_{n}^{s}, \end{equation}
where the right hand side is part of \eqref{rnsn} and follows from the requirement $\Hd E = s$. For large enough $n$, the condition \eqref{form52} is already more restrictive than $j > 2^{n_{\alpha} + m_{s}}$, so the latter condition can be simply disregarded. Now, recalling that $r_{n} = 3^{-2^{m_{s} + n}}$, it is easy to check, using the right hand side of \eqref{form52}, that the following inequalities hold for all sufficiently large $n$:
\begin{displaymath} 3^{n - 1}r_{n} < r_{n}^{1 - \sigma} \quad \text{and} \quad s_{n} > r_{n}^{s + \sigma}. \end{displaymath}
Then, for such $n$, the number of indices in $\calJ_{n}$ is at least
\begin{equation}\label{form40} |\calJ_{n}| \geq \log_{3} \frac{1}{r_{n}^{1 - \sigma}} - \log_{3} \frac{1}{r_{n}^{s + \sigma}} = \log_{3} r_{n}^{s - 1 + 2\sigma} = \log_{3} r_{n}^{(s - 1)/2}. \end{equation}
Now, let $N = N(M,s) \in \N$ be the smallest integer satisfying $N(1 - s)/2 \geq M$, so that 
\begin{equation}\label{form53} 3^{-N|\calJ_{n}|} \leq 3^{-N\log_{3} r_{n}^{(s - 1)/2}} \leq r_{n}^{M} \end{equation}
for all $n$ so large that \eqref{form40} holds. Finally, choose $\alpha = \alpha(s,M) \in (0,1)$ to be a number satisfying
\begin{equation}\label{form45} (1 - \alpha)^{\tfrac{2}{3} - \sigma}(1 + 2\alpha)^{\tfrac{1}{3} + \sigma} = 3^{-N}. \end{equation}
It then follows from \eqref{form42}, \eqref{form41}, and \eqref{form53} that if $I \in \calG_{n} \subset \calS_{n}$, and $n$ is sufficiently large, then
\[ |f(I)| = \Theta_{n}(I)|I| \leq r_{n}^{1 + M}.\]
In particular, $\calG_{n}\subset \calS_{n}^{M}$ and so $G_{n}^{c} \subset B_{n}$. Thus, 
\begin{displaymath} B = [0,1) \setminus G = \limsup_{n \to \infty} G_{n}^{c} \subset \limsup_{n \to \infty} B_{n}, \end{displaymath}
and, it suffices to verify that the Lebesgue measure of the set on the right hand side is zero. But this is a straightforward combination of the Borel-Cantelli lemma, the Chernoff bound \eqref{form44}. The proof of the lemma is complete. \end{proof}

Now, we may prove the main result:
\begin{proof}[Proof of Theorem \ref{mainTechnical}] Recall the notation $K := [0,1] \times E^{d - 1}$, and let $G = G_{M} \subset [0,1]$ be the set constructed in the previous lemma (for $M$ to be chosen very soon), and assume that $\alpha$ is so close to one that $\calL([0,1] \setminus G) = 0$. Then $G \times E^{d - 1}$ has full measure with respect to $\calL \times \nu$, for any Radon measure $\nu$ on $E^{d - 1}$. It remains to show that if $M = M_{d,\epsilon,s} \in \N$ is chosen large enough, then $\Hd F(G \times E^{d - 1}) \leq \epsilon$. Recalling the definition of $G$, the estimate $\Hd F(G \times E^{d - 1}) \leq \epsilon$ follows, if we manage to show that
\begin{equation}\label{form59} \dim F\left(\bigcap_{n=k}^{\infty}G_{n} \times E^{d-1}\right) \leq \epsilon, \qquad k \in \N. \end{equation}
To this end, fix $\delta > 0$, $n \geq k$, and $Q = I \times J_{2} \times \cdots \times J_{d}$ with $I \in \calS_{n}^{M} \subset \calS_{n}$ and $J_{j}\in \calI_{n}$, for $2 \leq j \leq d$. Fix any $x,y \in K \cap Q$ so that $x_{1}$ is the left endpoint of $I$ and $y_{1}$ is the right endpoint of $I$. Then, fix an arbitrary point $z \in K \cap Q$; it follows that $x_{1} \leq z_{1} \leq y_{1}$, and either $z_{1} - x_{1} \sim \diam(Q) \geq |z - x|$ or $y_{1} - z_{1} \sim \diam(Q) \geq |z - y|$. Assume, for instance, that the former holds. Then, by Lemma \ref{fcomparable}, and the doubling of $\mu$, we infer that
\begin{align*} |F(z) - F(x)| &\lesssim_{\mu} \mu([x_{1}, x_{1}+|z - x|)) \lesssim_{\mu} \mu([x_{1}, z_{1})) \leq \mu(I) \leq r_{n}^{1 + M}, \end{align*}
using also the definition of "$I \in \calS_{n}^{M}$" in the last inequality. The same estimate also holds for $x,y$ in place of $x,z$, and consequently $\diam F(Q) \leq C_{\mu}r_{n}^{1 + M}$ for some constant $C_{\mu} \geq 1$ depending on $\mu$. Finally, by choosing $n = n_{\delta} \geq k$ such that $C_{\mu}r_{n}^{1 + M} < \delta$, we have
\begin{align*} \calH^{\epsilon}_{\delta} \left[ F\left(\bigcap_{n = k}^{\infty} G_{n}\times E^{d-1} \right) \right] &\leq \calH^{\epsilon}_{\delta}(F(G_{n} \times E^{d-1})) \leq |\calS_{n}||\calI_{n}|^{d - 1} \cdot (C_{\mu}r_{n}^{1 + M})^{\epsilon}\\
  &\lesssim_{\mu} r_{n}^{-(1+(d-1)s) + (1 + M)\epsilon}.
\end{align*}
Picking $M = M_{d,\epsilon,s} \in \N$ so large that $(1 + M)\epsilon - (1+(d-1)s) > 0$, the right hand side tends to zero as $n \to \infty$. This proves \eqref{form59}, and the theorem. \end{proof}

\appendix

\section{A lemma on Hausdorff measures and products}\label{A}

The next lemma implies, in particular, that if $0 \leq s \leq d - 1$, and $E \subset \R^{d - 1}$ is a Borel set with $0 < \calH^{s}(E) < \infty$, then the product measure $\calL^{1}|_{[0,1]} \times \calH^{s}|_{E}$ is equivalent to the restriction of $(1 + s)$-dimensional Hausdorff measure on $[0,1] \times E$. This was required to deduce Theorem \ref{mainHigherDimension} from Theorem \ref{mainTechnical}.

\begin{lemma} Assume that $A_{1} \subset \R^{d_{1}}$ and $A_{2} \subset \R^{d_{2}}$ are Borel sets, and $s_{1} \in (0,d_{1}]$, $s_{2} \in (0,d_{2}]$ are numbers such that $0 < \calH^{s_{1}}(A_{1}) < \infty$ and $0 < \calH^{s_{2}}(A_{2}) < \infty$. Assume, moreover, that $A_{1}$ is $s_{1}$-Ahlfors-David regular. Then the measures $\calH^{s_{1}}|_{A_{1}} \times \calH^{s_{2}}|_{A_{2}}$ and $\calH^{s_{1} + s_{2}}|_{A_{1} \times A_{2}}$ are mutually absolutely continuous. \end{lemma}

\begin{proof} The absolute continuity $\calH^{s_{1}}|_{A_{1}} \times \calH^{s_{2}}|_{A_{2}} \ll \calH^{s_{1} + s_{2}}|_{A_{1} \times A_{2}}$ does not require regularity from $A_{1}$ and follows, for instance, from Corollary 5.9 in Falconer's book \cite{Fa}. 

To prove the converse direction of absolute continuity, fix a cube $Q = Q_{1} \times Q_{2} \subset \R^{d_{1} + d_{2}}$, where $Q_{1} \subset \R^{d_{1}}$ and $Q_{2} \subset \R^{d_{2}}$ are cubes of the same side-length, centred at $A_{1}$ and $A_{2}$, respectively, with $\diam(Q_{j}) \leq \diam(A_{j})$. Then, we claim that
\begin{equation}\label{form60} \frac{\calH^{s_{1} + s_{2}}|_{A_{1} \times A_{2}}(Q)}{[\calH^{s_{1}}|_{A_{1}} \times \calH^{s_{2}}|_{A_{2}}](Q)} \lesssim_{A_{1}} 1, \end{equation}
where the implicit constants only depend on the ambient dimensions $d_{1},d_{2}$ (which we treat as absolute constants in the notation), and the regularity constants of $A_{1}$. This will imply, according to Theorem 2.12(3) in \cite{Ma}, that $\calH^{s_{1} + s_{2}}|_{A_{1} \times A_{2}} \ll  \calH^{s_{1}}|_{A_{1}} \times \calH^{s_{2}}|_{A_{2}}$. 

Fix $\delta > 0$ and cover $A_{2} \cap Q_{2}$ by balls $\{B_{j}\}_{j \in J}$ so that 
\begin{displaymath} r_{j} := \diam(B_{j}) \leq \delta \leq \diam(A_{1} \cap Q_{1}) \end{displaymath}
(note that $\diam(A_{1} \cap Q_{1}) > 0$, because $Q_{1}$ is centred at $A_{1}$, and $A_{1}$ is $s_{1}$-Ahlfors-David regular with $s_{1} > 0$), and
\begin{displaymath} \sum_{j \in J} r_{j}^{s_{2}} \lesssim \calH^{s_{2}}(A_{2} \cap Q_{2}). \end{displaymath}
Then, for every index $j \in J$, cover $A_{1} \cap Q_{1}$ by $\lesssim \calH^{s_{2}}(A_{1} \cap Q_{1})r_{j}^{-s_{1}}$ balls of diameter $r_{j}$; this is possible by the Ahlfors-David regularity of $A_{1}$, and we denote these balls by $\{B^{j}_{i}\}_{i \in I_{j}}$. Now, the sets $B_{j} \times B^{j}_{i}$, with $j \in J$ and $i \in I_{j}$, cover $[A_{1} \times A_{2}] \cap Q$, and
\begin{align*} \sum_{j \in J} \sum_{i \in I_{j}} \diam(B_{j} \times B^{j}_{i})^{s_{1} + s_{2}} & \sim \sum_{j \in J} r_{j}^{s_{1} + s_{2}} \card (I_{j})\\
& \lesssim \calH^{s_{1}}(A_{1} \cap Q_{1}) \sum_{j \in J} r_{j}^{s_{2}}\\
& \lesssim \calH^{s_{1}}(A_{1} \cap Q_{1})\calH^{s_{2}}(A_{2} \cap Q_{2})\\
& = [\calH^{s_{1}}|_{A_{1}} \times \calH^{s_{2}}|_{A_{2}}](Q). \end{align*} 
This proves \eqref{form60}, and the lemma. \end{proof}

\end{document}